\documentclass[11pt,a4paper,dvipsnames,eqrno]{article}
\usepackage[utf8]{inputenc}

\title{\vspace*{-0.5cm}\scshape Combinatorics of slices of cubes}
\author{\scshape Marie-Charlotte Brandenburg \and \scshape Chiara Meroni }

\date{}

\usepackage{a4wide}
\usepackage{amssymb,mathrsfs}
\usepackage{amsmath}
\usepackage{euscript}
\usepackage{amsthm}
\usepackage{mathtools}
\usepackage{amsopn}
\usepackage{stackrel}
\usepackage[all]{xy}
\usepackage{setspace}
\usepackage{float}
\usepackage{blkarray}
\usepackage[dvipsnames]{xcolor}
\usepackage[pdftex, colorlinks]{hyperref}
\hypersetup{linkcolor=MidnightBlue, citecolor=WildStrawberry, urlcolor=PineGreen}
\usepackage{pgfplots}
\usepackage[margin=0.6cm]{caption}
\usepackage{subcaption}
\pgfplotsset{width=7cm, compat=1.10}
\usepgfplotslibrary{fillbetween}
\usepackage{tikz}
\usepackage{tikz-cd}
\usetikzlibrary{matrix}
\usetikzlibrary{intersections}
\usepackage{mathdots}
\usepackage[nameinlink]{cleveref}
\usepackage{verbatim}
\usepackage{enumerate}
\usepackage{multirow}
\usepackage{blkarray}
\usepackage{subcaption}
\usepackage[title]{appendix}

\usepackage{algorithm}
\usepackage{algpseudocode} 
\renewcommand{\Require}[1]{\algorithmicrequire \ #1}
\renewcommand{\Ensure}[1]{\newline\algorithmicensure \ #1}
\renewcommand{\Return}[1]{\newline\algorithmicreturn \ #1}
\algrenewcommand{\algorithmicrequire}{\textsc{Input:}}
\algrenewcommand{\algorithmicensure}{\textsc{Output:}}

\usepackage[backend=bibtex,bibencoding=utf8,style=alphabetic]{biblatex}
\AtBeginBibliography{\small}
\AtEveryBibitem{%
\ifentrytype{article}{
    \clearfield{issn}%
    \clearfield{isbn}%
}{}
\ifentrytype{inproceedings}{
    \clearfield{issn}%
    \clearfield{isbn}%
}{}
\ifentrytype{inbook}{
    \clearfield{issn}%
    \clearfield{isbn}%
}{}
}
\usepackage{xurl}

\colorlet{figblue}{MidnightBlue!40}
\colorlet{figorange}{orange!50}
\colorlet{figgreen}{PineGreen!60!ForestGreen}

\usetikzlibrary{arrows.meta}

\usepackage{todonotes}

\newcommand{\central}[1]{\textcolor{figblue}{#1}}

\newcommand{\R}{\mathbb{R}}

\newcommand{\N}{\mathbb{N}}

\DeclareMathOperator{\spn}{span}

\DeclareMathOperator{\sgn}{sgn}

\newtheoremstyle{customplain}      % Name
  {1em}                          % Space above
  {1em}                            % Space below
  {\itshape}                       % Body font (italic)
  {}                               % Indent amount
  {\bfseries}                      % Theorem head font (bold)
  {.}                              % Punctuation after theorem head
  {.5em}                           % Space after theorem head
  {}                               % Head spec

\newtheoremstyle{customdefinition} % Name
  {1em}                          % Space above
  {1em}                            % Space below
  {\normalfont}                    % Body font (upright)
  {}                               % Indent amount
  {\bfseries}                      % Theorem head font (bold)
  {.}                              % Punctuation after theorem head
  {.5em}                           % Space after theorem head
  {}  

\theoremstyle{customplain}
\newtheorem{theorem}{Theorem}[section]
\newtheorem*{theorem*}{Results}
\newtheorem{corollary}[theorem]{Corollary}
\newtheorem{lemma}[theorem]{Lemma}
\newtheorem{proposition}[theorem]{Proposition}

\newtheorem{question}{Question}
\newtheorem{conjecture}{Conjecture}
\theoremstyle{customdefinition}

\newtheorem{examplex}[theorem]{Example}
\newenvironment{example}
{\pushQED{\qed}\examplex}
{\popQED\endexamplex}
\newtheorem{remark}[theorem]{Remark}

\newcommand{\ma}{\begin{pmatrix}}
\newcommand{\trix}{\end{pmatrix}}
\newcommand{\sma}{\left(\begin{smallmatrix}}
\newcommand{\strix}{\end{smallmatrix}\right)}

\DeclareMathOperator{\vertices}{vert}

\usepackage{parskip}
\allowdisplaybreaks

\makeatletter
\def\keywords{\xdef\@thefnmark{}\@footnotetext}
\makeatother

\makeatletter
\def\mscclasses{\xdef\@thefnmark{}\@footnotetext}
\makeatother

\begin{document}

\maketitle

\begin{abstract}
    We present a complete computational classification of the combinatorial types of hyperplane sections, or slices, of the regular cube up to dimension six. For each dimension, we determine the exact number of distinct combinatorial types. When restricted to slices through the origin, our computations extend to dimension seven. The classification combines combinatorial, algebraic, and numerical techniques, with all results certified.
    Beyond enumeration, we analyze the distribution of types by number of vertices, establish new theoretical results about the combinatorics of slices of cubes,
    and propose conjectures motivated by our computational findings.
\end{abstract}

\section{Introduction}
The cube is a fundamental yet intricate mathematical object, discussed extensively in numerous surveys and books \cite{Saks1993,Zong2005,Zong2006,Nayar2023}.
In this article, we focus on \emph{slices}, i.e., $(d-1)$-dimensional polytopes obtained as the intersection of the $d$-dimensional cube $C_d = [-1,1]^d$ with a hyperplane.
Because cubes appear throughout mathematics, their slices are likewise of broad interest, ranging across combinatorics, discrete geometry, optimization, machine learning, metric and convex geometry, number theory, and analysis. This article presents a complete computational classification of the combinatorics of slices of the cube up to dimension $6$, together with a collection of theoretical results and conjectures.
For this classification we combine several computational methods from combinatorics and numerical algebraic geometry, grounded in algebraic statistics and differential topology.

\begin{theorem}\label{thm:main_affine}
    The number of combinatorial types of $(d-1)$-dimensional hyperplane sections of the cube $C_d$ in dimensions $d = 3, 4, 5, 6$ is $4, 30, 344, 7346$ respectively.
\end{theorem}
\begin{figure}[!h]
    \centering
    \includegraphics[width=0.288\linewidth]{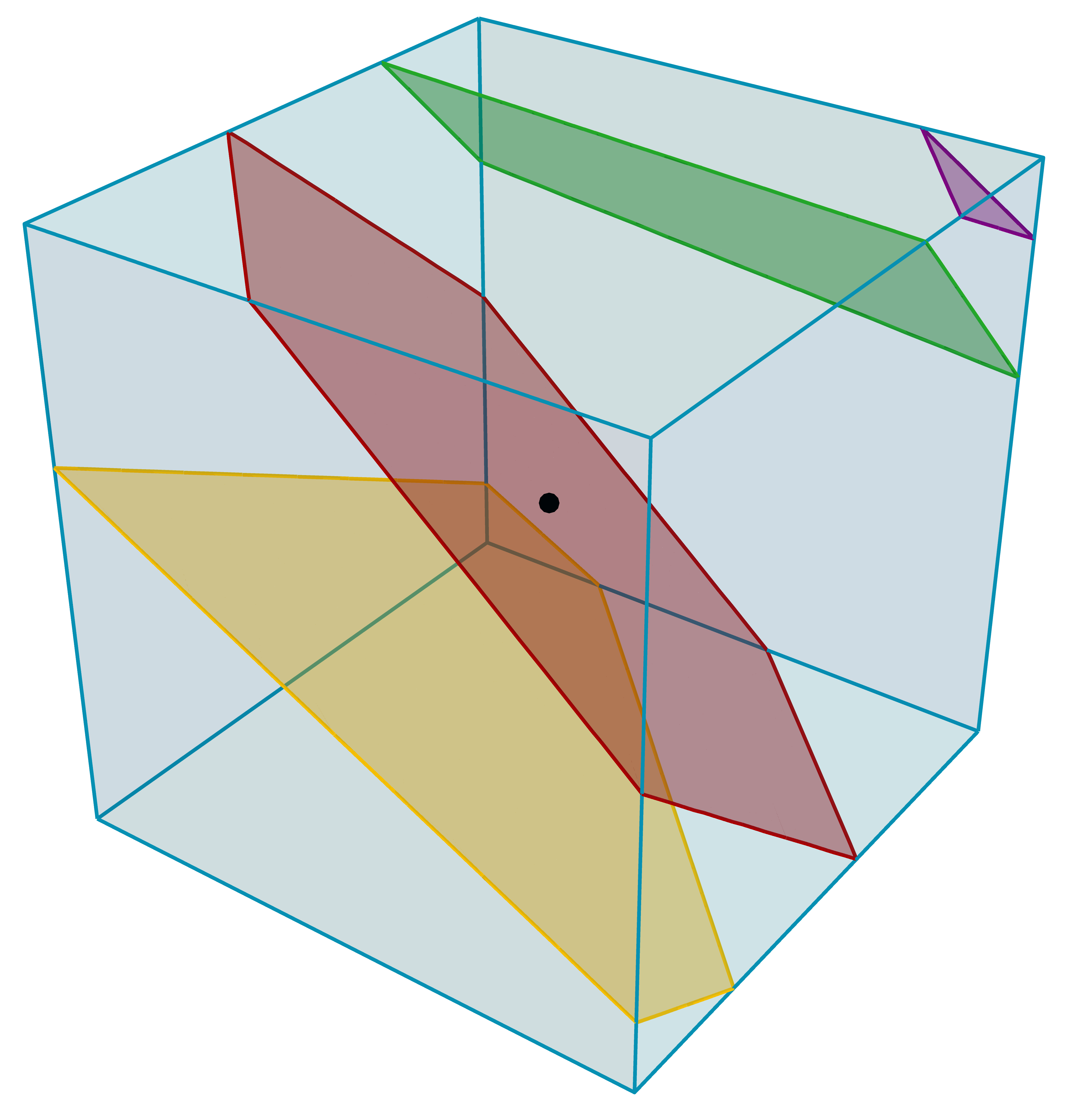}
    \caption{The cube $C_3$ and all its combinatorial types of slices: a purple triangle, a green quadrilateral, a yellow pentagon, a red hexagon.}
    \label{fig:intro}
\end{figure}

When restricting to slices through the origin, our computation extends to dimension $7$. Indeed, as shown below, the complexity of computing affine slices of $C_d$ matches that of computing central slices of $C_{d+1}$.
Our results for central slices are summarized in the following theorem.

\begin{theorem}\label{thm:main_central}
    The number of combinatorial types of $(d-1)$-dimensional hyperplane sections through the origin of the cube $C_d$ in dimensions $d = 3, 4, 5, 6, 7$ is $2, 6, 23, 133, 1657$ respectively.
\end{theorem}

\textbf{Combinatorial types.}
The number of combinatorial types of $(d-1)$-slices of the $d$-dimensional cube $C_d$ is well known to be $4$ for $d = 3$, see \Cref{fig:intro}. For $d = 4$, Evers appears to have shown that this number is at least 30~\cite{evers2010hyperebenenschnitte,Frank2012}, although the result was never published. In this article, we present exact counts for $d = 4, 5, 6$, and confirm that Evers' upper bound is indeed tight. 
Independently, Nakamura and Sasaki provided a count for $d = 4$ (see~\cite{Nakamura1980}; in Japanese, not available online and not peer-reviewed), which was extended to $d = 5$ in the work of Fukuda et al.~\cite{Fukuda1997}. However, their numbers differ from those presented here -- we reproduce their computations and explain the discrepancy, which stems from different definitions of combinatorial types.

\textbf{Number of vertices.}
The maximum number of vertices of affine slices of $C_d$ was determined by O'Neil, and more recently De Loera et al.\ identified the sequences of vertex numbers occurring among all affine slices for $d = 6$ \cite{ONeil1971,deloera2025numberverticeshyperplanesection}. In this article, we extend these results by providing the number of combinatorial types corresponding to each possible vertex count, and we prove that O'Neil's bound can be attained not only by affine slices but even by slices through the origin.

\textbf{Computational methods.}
In order to perform these computations we combine combinatorial, algebraic and numerical techniques, supported by several software packages.
\cite{Brandenburg2025} provides algorithms to compute combinatorics and volumes of slices for general polytopes, with implementations in \texttt{SageMath} \cite{sagemath}, but these are infeasible for the computationally demanding $6$-dimensional cube. To overcome these limitations, we draw on the theory of the maximum-likelihood degrees of very affine varieties \cite{Huh2013}, together with recent generalizations to hypersurface arrangements \cite{Reinke2024}. Key \texttt{Julia} \cite{Julia-2017} packages include the recent \texttt{HypersurfaceArrangements.jl} \cite{Breiding2024}, which rests on Morse theory and on \texttt{HomotopyContinuation.jl} \cite{Breiding2018}, as well as \texttt{Oscar.jl} \cite{OSCAR,OSCAR-book} for combinatorial tasks such as intersecting and comparing polytopes. Although our approach relies on numerical methods, all results are numerically certified, ensuring that the floating-point computations correspond to exact solutions.

\paragraph{Outline.}
Given the extensive literature on slices of cubes, we conclude the introduction with a brief overview of known results.
In \Cref{sec:description_slices} we present computational and theoretical results on the combinatorics of slices, both affine and central. The main statements are \Cref{thm:upper_bound_vertices,thm:generic_central_slices}. \Cref{sec:algo} describes the computational framework and provides additional results and conjectures (see, e.g., \Cref{conj:central_slices_only_cube}) suggested by the data, and also contains the proofs of \Cref{thm:main_affine,thm:main_central}. In \Cref{sec:graphs} we compare our computations with those of Fukuda et al.~\cite{Fukuda1997}, leading to further open questions. Finally, we provide two appendices: \Cref{app:distribution} displays plots of the distributions of the combinatorial types of slices, and \Cref{app:repo} describes the online repository \url{https://doi.org/10.5281/zenodo.17304584} where both our code and collected data are available.

\subsection*{Related work}
\textbf{Threshold functions.}
A key step in computing the combinatorics of cube slices is the computation of the \emph{threshold arrangement}. A \emph{linear threshold function} (or \emph{linear threshold device}) is a linear function whose zero set is a hyperplane that partitions the vertices of the cube, and a self-dual linear threshold function corresponds to a hyperplane through the origin. The parameter space of linear threshold functions is subdivided into regions induced by the \emph{threshold arrangement}, a hyperplane arrangement on which the vertex partition remains constant within each cell. This arrangement is closely tied to the task of classifying slices of the cube and its computation is therefore a crucial step of the algorithms developed in this article. Threshold functions play a fundamental role in the early theory of neural networks \cite{cover64_geometricalstatisticalproperties,Abelson1977,Wenzel2000,Ojha2000,} and remain relevant for modern neural network architectures \cite{Chen22:ThresholdNN,Ergen2023:GloballyOptNeural,Khalife2023}. Interestingly, they are also connected to topics in convex geometry and algebraic combinatorics, such as zonotopes~\cite{Montufar2015} and triangulations of the root polytope of type~A~\cite{Gutekunst2021}. Considerable computational effort has been devoted to determining the exact number of linear threshold functions (i.e., the number of cells of the threshold arrangement) for $d \leq 10$ \cite{Muroga1962,Winder1965,Muroga1970,Brysiewicz2023}, and it is known that the logarithm of this number grows asymptotically like $d^2$ \cite{ZUEV1992}.

\textbf{Volumes.}
Extensive tools have been developed to study the volumes of slices of cubes.
A classical problem concerns \emph{extremal slices}, i.e., those with maximal or minimal volume among all slices.
Hadwiger showed in the 1970s that the minimum volume of a slice through the center of symmetry is attained by a hyperplane parallel to a facet (\cite{Hadwiger1972}; in German), a result later reproved by Hensley \cite{Hensley1979}, who additionally conjectured the maximum volume of a slice. 
A few years later, Ball proved in his seminal work that the maximum-volume slice is orthogonal to a main diagonal of a $2$-face \cite{ball86_cubeslicing$rn$}. These results continue to influence modern research, inspiring variations such as the computation of extremal slices for normalized volumes
\cite{Aliev2020,Barany2022}, the description of the space of all extremal slices \cite{Ambrus2021}, and the study of extremal lower-dimensional slices of cubes \cite{Vaaler1979,Ball87,Ivanov2021}. Beyond extremal cases, one may also consider the volume as a function over all hyperplane sections of a cube. This has lead to formulas for the volume of general $(d-1)$-dimensional slices \cite{Marichal2008,Frank2012} or for specific subclasses \cite{Chakerian1991,Bartha2020}, and it has been shown in \cite{Berlow2022,Brandenburg2025} that the integral of any polynomial (such as the volume) over the slices is piecewise polynomial on the cells of a hyperplane arrangement. An active line of current research studies extremal slices of cubes cut by hyperplanes at distance $t$ from the origin \cite{Moody2013,Pournin2023a,Pournin2023,Ambrus2024,pournin2025deepsectionshypercube}, motivated by a conjecture of Milman -- reported by K\"onig and Koldobsky \cite{Koenig2011} -- on the location of such slices for specific ranges of $t$.
Further variations include slices with maximal perimeter \cite{Koenig2021}, slices maximizing different measures \cite{Zvavitch2008,Koenig2019}, and the volumes of slabs of cubes \cite{Marichal2008,Koenig2011}. 

\textbf{Related volume questions on polytopes.}
One way to study the volumes of all central slices of a cube, or more generally of any polytope, is through its \emph{intersection body}, an object closely connected to the classical \emph{Busemann–Petty problem}~\cite{BusemannPetty}.
This problem asks whether, whenever the $(d-1)$-dimensional volumes of the slices of one centrally symmetric convex body are bounded above by those of another, the same inequality holds true for their $d$-dimensional volumes. A simple high-dimensional counterexample compares slices of cubes with slices of balls. Indeed, the Busemann-Petty problem holds only in dimensions up to $4$ \cite{GKS99:BusemannPetty}, and the quest to establish this fact led to the development of the theory of intersection bodies \cite{Lutwak88:IntBodies,Gardner94:IntBodies,Koldobsky98:IntBodies}. Since then, the study of intersection bodies of cubes has attracted independent interest; see, e.g., \cite{aliev08_siegelslemmasum,Aliev2020,Berlow2022}. 
The negative resolution of the Busemann-Petty problem also inspired Bourgain's slicing conjecture -- now a theorem \cite{KlartagLehec25:AffirmativeBourgain} -- which asserts the existence of a universal constant in the inequality for $d$-dimensional volumes. Analogous questions can be formulated for lattice point counts \cite{FreyerHenk24:DiscreteSlicing}. 

\textbf{Variations on cubes.}
The cube can also be seen as the unit ball of the $\ell_\infty$-norm, and most of the questions discussed above have natural analogues for unit balls of $\ell_p$-norms, more in general \cite{MeyerPajor88:lpball,Oleszkiewicz03:lp,Koldobsky05:ConvexGeom,konig25:maxSectionslp}. 

On another note, instead of changing the norm, one can change the notion of volume by studying the \emph{discrete volume}, that is, the number of lattice points in a slice of a polytope. 
This quantity is not well understood for general polytopes, but for slices of cubes both a closed formula for the number of lattice points~\cite{Abel2018} and a combinatorial description of their Ehrhart polynomials~\cite{Ferroni2024} are known.

Another point of view comes from probability and concerns the number of faces of a slice. For central slices, there is a closed formula for the expected number of vertices of a random $k$-dimensional slice of $C_d$, together with a lower bound on the expected number of faces of any dimension \cite{Lonke2000}.
For affine slices, the expected number of vertices of a random $k$-dimensional slice -- with a suitable distribution -- is $2^k$, independent of the dimension of the cube \cite{Swan2016}.

\subsection*{Acknowledgments}
We would like to thank Jes\'us De Loera, Simon Telen, and Kexin Wang for many fruitful conversations. We are especially grateful to Saiei-Jaeyeong Matsubara-Heo for locating the paper \cite{Nakamura1980} and generously translating it for us.
We also thank Benjamin Schr\"oter for pointing us to \cite[Example 23]{PinedaVillavicencio2022}. CM is supported by Dr. Max R\"ossler, the Walter Haefner Foundation, and the ETH Z\"urich Foundation. MB is supported by the SPP 2458 ``Combinatorial Synergies'', funded by the Deutsche Forschungsgemeinschaft (DFG, German Research Foundation).

\section{Description of the slices}\label{sec:description_slices}

In this section, we present computational and theoretical results on slices of the centrally symmetric cube $C_d = [-1,1]^d$. Here, a \emph{slice} is any $(d-1)$-dimensional polytope which can be obtained as the intersection of $C_d$ with a hyperplane. Lower-dimensional polytopes obtained in this way are disregarded. A hyperplane is said to be \emph{generic} with respect to $C_d$ if it does not intersect any vertex of $C_d$. A slice of $C_d$ is a \emph{generic slice} if it arises as the intersection of $C_d$ with a generic hyperplane. A slice of $C_d$ is a \emph{central slice} if the defining hyperplane contains the origin. When we wish to emphasize that the hyperplane need not contain the origin, we refer to the slice as \emph{affine}.

\begin{table}[ht]
\centering
\begin{tabular}{l|ccccccc}
                       & $d=$ \hspace{-1.2em} & 2 & 3 & 4  & 5   & 6 & 7 \\ 
                       \hline
affine slices          & & 1     & 4 & 30 & 344 & 7346 & - \\
generic affine slices  & & 1     & 4 & 12 & 58 & 554 & - \\
central slices         & & 1     & 2 & 6  & 23 & 133 & 1657 \\
generic central slices & & 1     & 2 & 3  &  7  & 21 & 135
\end{tabular}
\caption{Number of combinatorial types of $(d-1)$-dimensional slices of the cube $C_d=[-1,1]^d$. For the $C_7$, the affine slices computation is out of reach.}
\label{tab:comb-types}
\end{table}

The \emph{combinatorial type} of a polytope is the isomorphism class of its face lattice. Although any full-dimensional polytope admits infinitely many slices from a metric point of view, only finitely many distinct combinatorial types occur. Counting slices according to their combinatorial type is therefore a natural and finite classification strategy.
\Cref{tab:comb-types} displays the numbers of combinatorial types of slices of the cube $C_d$ for dimensions $d\leq 7$. All results for $d = 4, 5, 6, 7$ are new. We emphasize that our computations yield a complete list of all combinatorial types, whereas \Cref{tab:comb-types} merely records their counts. A complete list of representatives of each combinatorial type of slice is available at
\begin{center}
    \url{https://doi.org/10.5281/zenodo.17304584}
\end{center}

\subsection{Generic central slices}\label{sec:generic-central-slices}

When restricting to central slices, and in particular generic central slices, the rigidity of the cube can be exploited to characterize combinatorial types. 
Indeed, two generic central slices of the cube are combinatorially equivalent if and only if the corresponding hyperplanes induce the same partition of the sets of vertices of the cube, up to its symmetries. We now formalize this observation. 
A central hyperplane $H$ is determined by its normal vector $u$ (unique up to rescaling), and we write $H = u^\perp$. Every vector $u\in \R^d$ induces a partition of the vertices of $C_d$, grouping them into those with nonnegative scalar product with $u$ and those with negative scalar product. 

A \emph{signed permutation} $\pi$ of the set $[d]$ is a permutation of the set $\{-d,-d+1,\ldots,-1,1,\ldots,d\}$ satisfying $\pi(-i)=-\pi(i)$. With a slight abuse of notation, we let $\pi$ act on vectors in $\R^d$ by
\[
\pi(v) = \pi(v_1,\ldots,v_d) = (v_{\pi(1)}, \ldots, v_{\pi(d)}),
\]
where we define $v_{-i} = -v_i$. We denote by $B_d$ the group of signed permutations of $[d]$.

\begin{lemma}\label{lem:facets_generic_central_slices}
    For any $u\in \R^d$, the hyperplane $u^\perp$ either intersects all the facets of $C_d$ (not necessarily in their relative interiors) or intersects all but two parallel facets. In the latter case, $C_d \cap u^\perp$ is combinatorially equivalent to $C_{d-1}$.
\end{lemma}
\begin{proof}
    Assume that $u^\perp$ does not intersect a facet of $C_d$, which without loss of generality we assume to be contained in the affine hyperplane $H_{+d} = \{x \in \R^d \mid x_{d}=1\}$. By symmetry, $u^\perp$ then does not intersect the parallel facet contained in $H_{-d} = \{x \in \R^d \mid x_{d}=-1\}$. However, every edge of the cube that is not contained in one of these two facets has exactly one vertex in $H_{+d}$ and one in $H_{-d}$, so each such edge meets $u^\perp$. This proves the claim.
\end{proof}

This simple observation, relying on the rigid structure of the regular cube, allows us to reconstruct $C_d$ (up to symmetry) from the combinatorics of a given slice.
\begin{theorem}\label{thm:generic_central_slices}
    Let $u_1^\perp, u_2^\perp\subset \R^d$ be hyperplanes generic with respect to $C_d$. Then, $C_d \cap u_1^\perp$ is combinatorially equivalent to $C_d \cap u_2^\perp$ if and only if there exists a signed permutation $\pi \in B_d$ such that $u_1^\perp$ and $\pi(u_2)^\perp$ intersect the same edges of $C_d$.
\end{theorem}
\begin{proof}
    Assume first that there exists a signed permutation $\pi \in B_d$ such that the partition of the vertices of $C_d$ induced by $u_1$ coincides with the partition induced by $\pi(u_2)$.
    Then, by \cite[Lemma 2.4]{Berlow2022}, the slices $C_d \cap u_1^\perp$ and $C_d \cap \pi(u_2)^\perp$ are combinatorially equivalent. Moreover, the symmetries of the cube imply $C_d = \pi(C_d)$, and therefore $C_d\cap \pi(u_2)^\perp$ is affinely isomorphic, and hence combinatorially equivalent, to $C_d \cap u_2^\perp$. This proves one direction of the statement (note that we did not use the fact that the slices are central).

    Conversely, assume that $C_d \cap u_1^\perp$ is combinatorially equivalent to $C_d \cap u_2^\perp$. By the genericity of these hyperplanes, \Cref{lem:facets_generic_central_slices} implies that $C_d \cap u_1^\perp$ has either $2(d-1)$ facets or $2d$ facets. 
    If $C_d \cap u_1^\perp$ has $2(d-1)$ facets, then for each $i=1,2$, the hyperplane $u_i^\perp$ partitions the vertices of $C_d$ into those lying in the facets
    \[
    \{x \in \R^d \mid x_{j_i}=1\}\quad\text{and}\quad \{x \in \R^d \mid x_{j_i}=-1\}.
    \]
    Let $\pi \in B_d$ be any signed permutation satisfying $\pi(j_2)=j_1$. It follows that $\pi(u_2)^\perp$ partitions the vertices of $C_d$ inside $\{x \in \R^d \mid x_{j_1}=1\}$ and $\{x \in \R^d \mid x_{j_1}=-1\}$ exactly as $u_1^\perp$ does, implying that $u_1^\perp$ and $\pi(u_2)^\perp$ intersect $C_d$ in the same set of edges in their relative interiors.
    
    It remains to consider the case where $C_d \cap u_1^\perp$ has $2d$ facets. We label the facets of $C_d$ as
    \[
    F_{\star j} = C_d \cap \{x \in \R^d \mid x_j = \star 1\}, \quad\star \in \{+,-\}.
    \]
    Let $M_1$ and $M_2$ be the vertex–facet incidence matrices of the two slices, each with $2d$ columns indexed by $[d] \cup -[d]$.
    Since the slices are combinatorially equivalent, there exist permutations $\pi$ of the columns and $\sigma:[n] \to [n]$ of the rows such that $(M_1)_{\sigma(a), \pi(b)} = (M_2)_{a,b}$.
    Because both slices are centrally symmetric, $\pi$ maps every pair of opposite integers $(k,-k)$ to another pair of opposite integers $(\pi(k),\pi(-k)) = (\ell,-\ell)$,
    so $\pi \in B_d$ indeed maps the columns of $M_2$ to those of $M_1$.

    Without loss of generality we may assume $\sigma = \operatorname{id}$. Suppose that $\pi(u_2)^\perp$ intersects an edge $e$ of $C_d$. Then $e \cap \pi(u_2^\perp)$ is a vertex of $C_d \cap \pi(u_2)^\perp$, corresponding to the $k^{\text{th}}$ row of $M_2$.
    From $(M_2)_{k,\star j} = (M_1)_{k,\pi(\star j)}$ we deduce
    \[
    e = \bigcap_{\substack{\star j \in [d] \cup -[d] \\ (M_2)_{k,\star j} = 1}} F_{\pi(\star j)} = \bigcap_{\substack{\star j \in [d] \cup -[d] \\ (M_1)_{k,\pi(\star j)} = 1}} F_{\pi(\star j)} \ .
    \]
    The rightmost expression is precisely the edge intersected by $u_1^\perp$ to produce the $k^{\text{th}}$ vertex of $C_d\cap u_1^\perp$. This holds for every edge of the slices, showing that $u_1^\perp$ and $\pi(u_2)^\perp$ intersect $C_d$ in exactly the same set of edges.
\end{proof}

\begin{example}\label{ex:hexagonal_slices}
    Let $u_1 = \frac{1}{\sqrt3}(1,1,1)$ and $u_2 = \frac{4}{\sqrt{29}}\bigl(1,-\frac{1}{2},-\frac{3}{4}\bigr)$. The corresponding slices are both hexagons, as shown in \Cref{fig:hexagons}, with vertex-facet incidence matrices equal (up to permutation) to
    \[
    M_1 = M_2 = 
    \begin{pmatrix}
        1 & 1 & 0 & 0 & 0 & 0 \\
        0 & 1 & 1 & 0 & 0 & 0 \\
        0 & 0 & 1 & 1 & 0 & 0 \\
        0 & 0 & 0 & 1 & 1 & 0 \\
        0 & 0 & 0 & 0 & 1 & 1 \\
        1 & 0 & 0 & 0 & 0 & 1
    \end{pmatrix}.
    \]
    The columns of $M_1$ are indexed by $F_{+1}, F_{-2}, F_{+3}, F_{-1}, F_{+2}, F_{-3}$, and the columns of $M_2$ are indexed by $F_{+1}, F_{+2}, F_{-3}, F_{-1}, F_{-2}, F_{+3}$. This yields the signed permutation $\pi \in B_3$ given by
    \[
    1 \mapsto 1, \quad 2 \mapsto -2, \quad 3 \mapsto -3.
    \]
    The vectors $u_1$ and $\pi(u_2) = \frac{4}{\sqrt{29}}\bigl(1,\frac{1}{2},\frac{3}{4}\bigr)$ induce the same partition of the vertices of $C_3$ into the two sets $\pm\{(1,1,1),(1,1,-1),(1,-1,1),(-1,1,1)\}$.
    Hence, the hyperplanes $u_1^\perp$ and $\pi(u_2)^\perp$ intersect exactly the same set of edges of $C_3$.
    \begin{figure}
        \centering
        \includegraphics[width=0.3\linewidth]{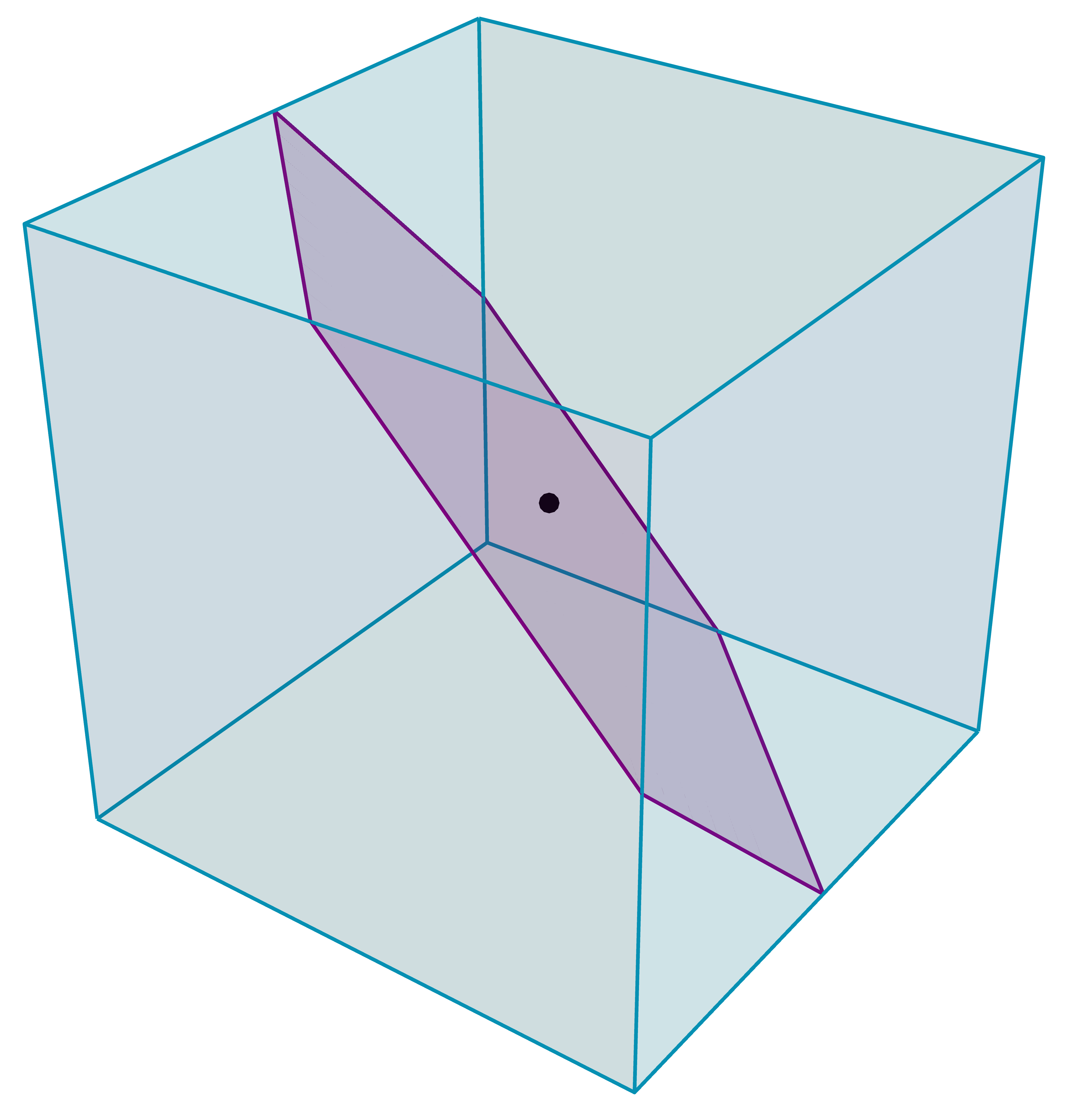}
        \qquad
        \includegraphics[width=0.3\linewidth]{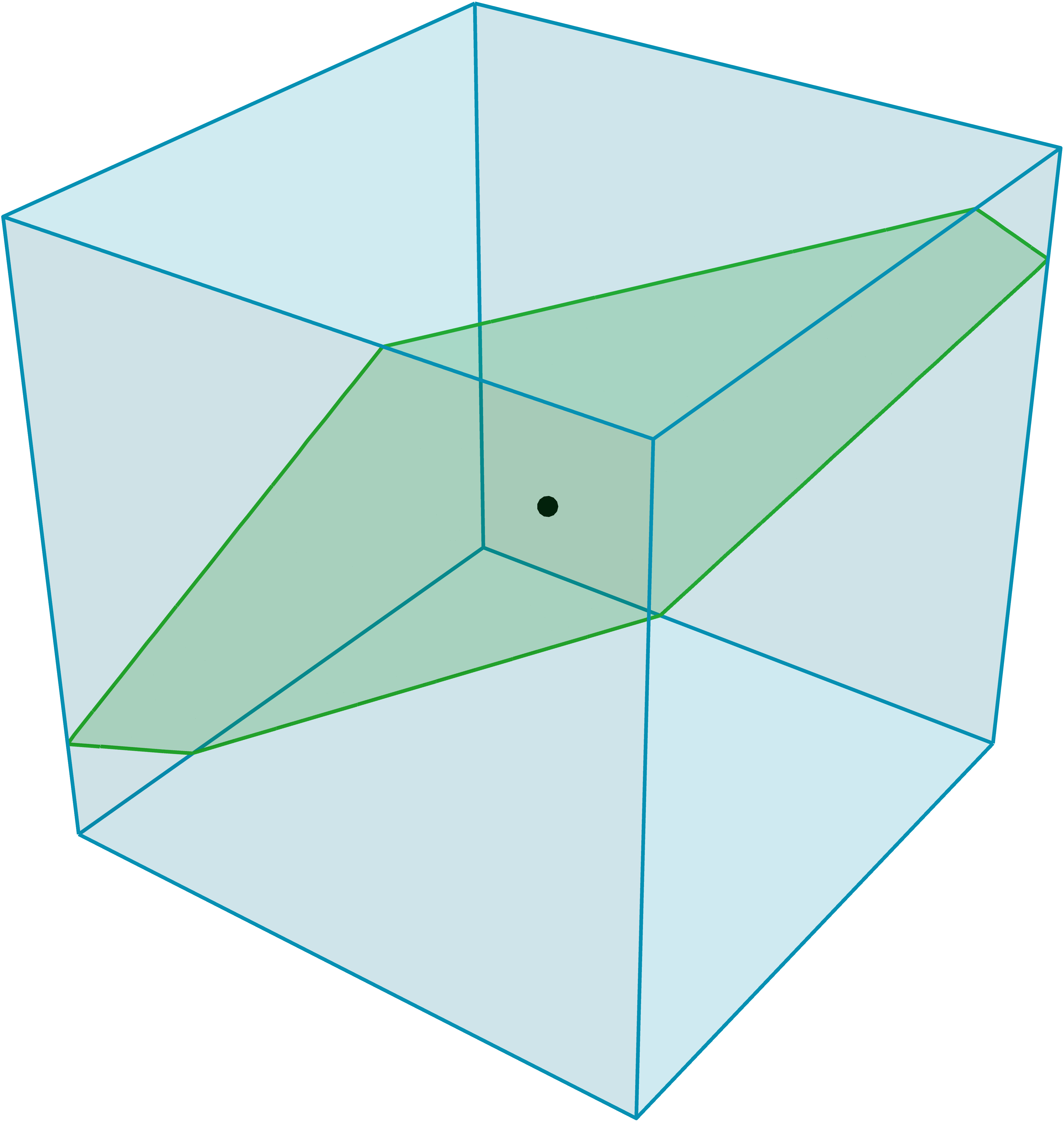}
        \caption{The slices $C_3\cap u_1^\perp$ (left) and $C_3\cap u_2^\perp$ (right) from \Cref{ex:hexagonal_slices}.}
        \label{fig:hexagons}
    \end{figure}
\end{example}

Despite one direction of \Cref{thm:generic_central_slices} not requiring the slices to be central (see first paragraph of its proof), the converse direction -- reconstructing the cube from a slice -- cannot be extended to noncentral slices, as the following example shows.

\begin{figure}[!h]
        \centering
        \includegraphics[width=0.3\linewidth]{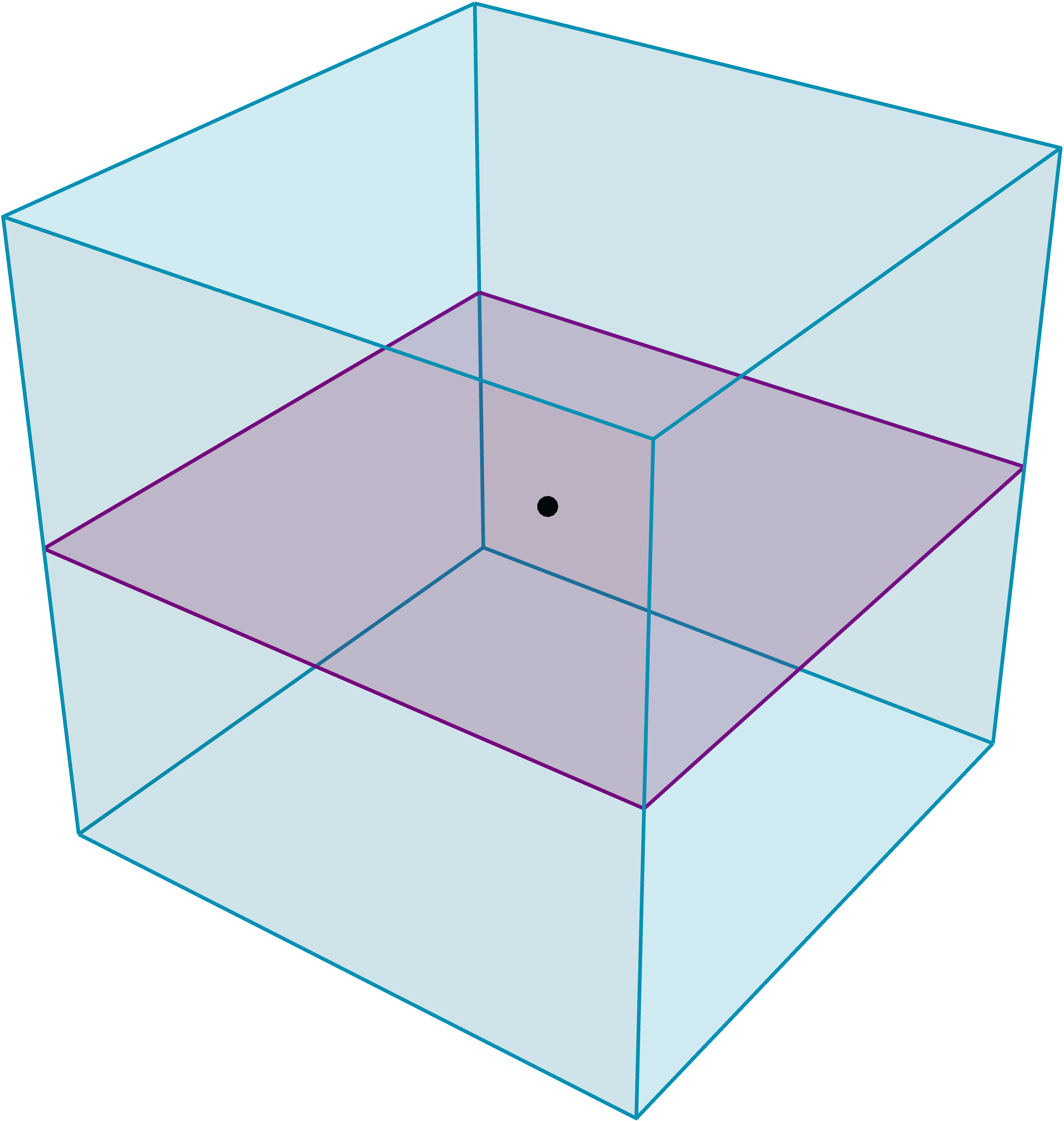}
        \qquad
        \includegraphics[width=0.3\linewidth]{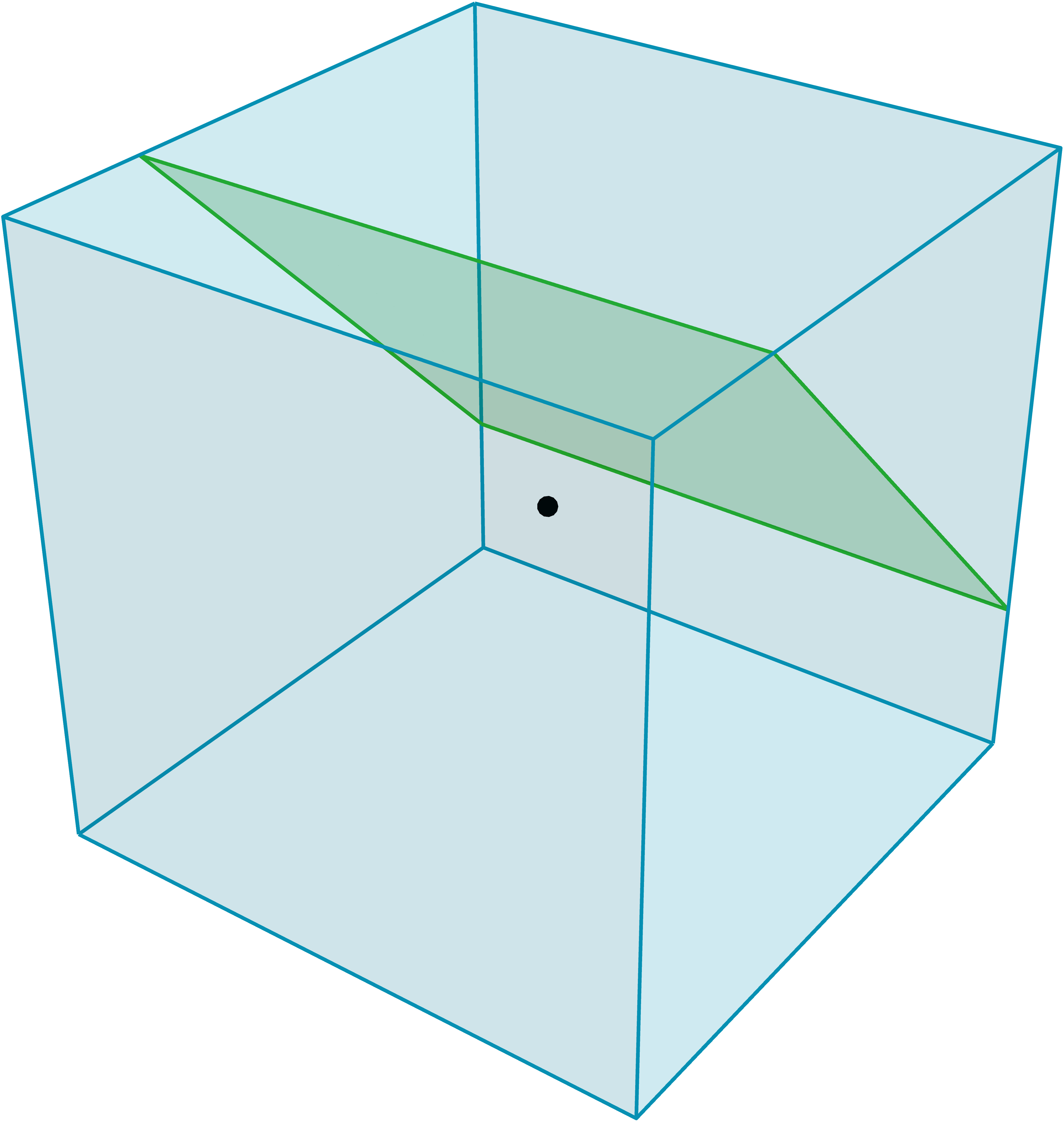}
        \caption{The slices $C_3\cap H_1$ (left) and $C_3\cap H_2$ (right) from \Cref{ex:quadrilateral_slices}.}
        \label{fig:quadrilaterals}
    \end{figure}
\begin{example}\label{ex:quadrilateral_slices}
    Let $H_1 = \{z=0\}$ and $H_2 = \{2y+2z=1\}$. The corresponding slices are both quadrilaterals, as displayed in \Cref{fig:quadrilaterals}, with vertex-facet incidence matrices equal (up to permutation) to
    \[
    M_1 = M_2 = 
    \begin{pmatrix}
        1 & 1 & 0 & 0 \\
        0 & 1 & 1 & 0 \\
        0 & 0 & 1 & 1 \\
        1 & 0 & 0 & 1
    \end{pmatrix}.
    \]
    The columns of $M_1$ are indexed by $F_{+1}, F_{+2}, F_{-1}, F_{-2}$, whereas those of $M_2$ are indexed by $F_{+1}, F_{+2}, F_{-1}, F_{+3}$. This defines a (partial) permutation of $\{\pm 1, \pm 2, \pm 3\}$ sending
    \[
    1\mapsto 1, \quad 2 \mapsto 2, \quad -1 \mapsto -1, \quad -2\mapsto 3.
    \]
    Such a mapping cannot be extended to a signed permutation, and indeed there is no signed permutation under which the two hyperplanes intersect the same set of edges of $C_3$.
\end{example}

This relation between the combinatorics of generic central slices of the cube and hyperplanes that partition its vertices leads naturally to the notion of threshold functions. A \emph{threshold function} on $d$ (or fewer) variables is a function
\[
f_{(w_0,w)}: \{-1,1\}^d \to \{0,1\}, \qquad
f_{(w_0,w)}(x) = 
\begin{cases}
    1 & \text{if } \langle w, x \rangle \geq w_0 \\
    0 & \text{otherwise} ,
\end{cases}
\]
where ${(w_0,w)} \in \R^{d+1}$. After a small perturbation of $(w_0,w)$ we may assume that $\langle w, x \rangle \neq w_0$ for all $x \in \{-1,1\}^d$. Geometrically, a threshold function $f_{(w_0,w)}$ records, for each vertex of the cube $C_d$, on which side of the hyperplane $\{x \in \R^d \mid \langle w,x \rangle = w_0\}$ the vertex lies.

The complement of a variable $x_i$ is $-x_i$. Threshold functions admit an $(S_2)^d$-action by independently complementing each variable, and an additional $(S_d)$-action by permuting the coordinates. An \emph{NP-equivalence class} of a threshold function is its equivalence class under these combined group actions. In other words, NP-equivalence classes correspond to orbits under the group $B_d$ of signed permutations. A threshold-function is \emph{self-dual} if it takes opposite values on complementary inputs, i.e., $f(-x) = 1-f(x)$ for all $x \in \{-1,1\}^d$. For this to hold, we must have
\[
    \langle w, -x \rangle \geq w_0 \iff 
    \langle w, x \rangle < w_0,
\]
which is equivalent to
\[
   w_0 + \langle w, x \rangle \leq 0 \text{ and }
   w_0 - \langle w, x \rangle >0 \quad \text{for all } x \in \{-1,1\}^d.
\]
This can occur only if $w_0 = 0$.
Thus every self-dual threshold function corresponds geometrically to a hyperplane through the origin.
Combining this with \Cref{thm:generic_central_slices} yields the following corollary.
\begin{corollary}
    Combinatorial types of generic central slices of the cube $C_d$ are in bijection with NP-equivalence classes of self-dual threshold functions on $d$ variables.
\end{corollary}

\subsection{Vertices of slices}

In \cite{deloera2025numberverticeshyperplanesection}, the authors give a complete characterization of the possible numbers of vertices of hyperplane sections of $C_d$ (regardless of the dimension of the intersection) for $d\leq 7$. We confirm their numbers and additionally determine which of these are realized as vertices of full-dimensional hyperplane sections, namely as vertices of what we call \emph{slices}. A natural question is the distribution of these vertex numbers: how many combinatorial types of slices have a fixed number of vertices?
This is shown in \Cref{fig:4_cube_vertices,fig:5_6cube_vertices}, which display these distributions for various types of slices of $C_4$, $C_5$, and $C_6$. Here, the light green bars represent all affine slices, the dark green bars represent generic affine slices, and the light and dark blue bars correspond to central and generic central slices, respectively.

\begin{figure}[ht]
\hspace{-2em}
    \centering
    \includegraphics[width=0.7\textwidth]{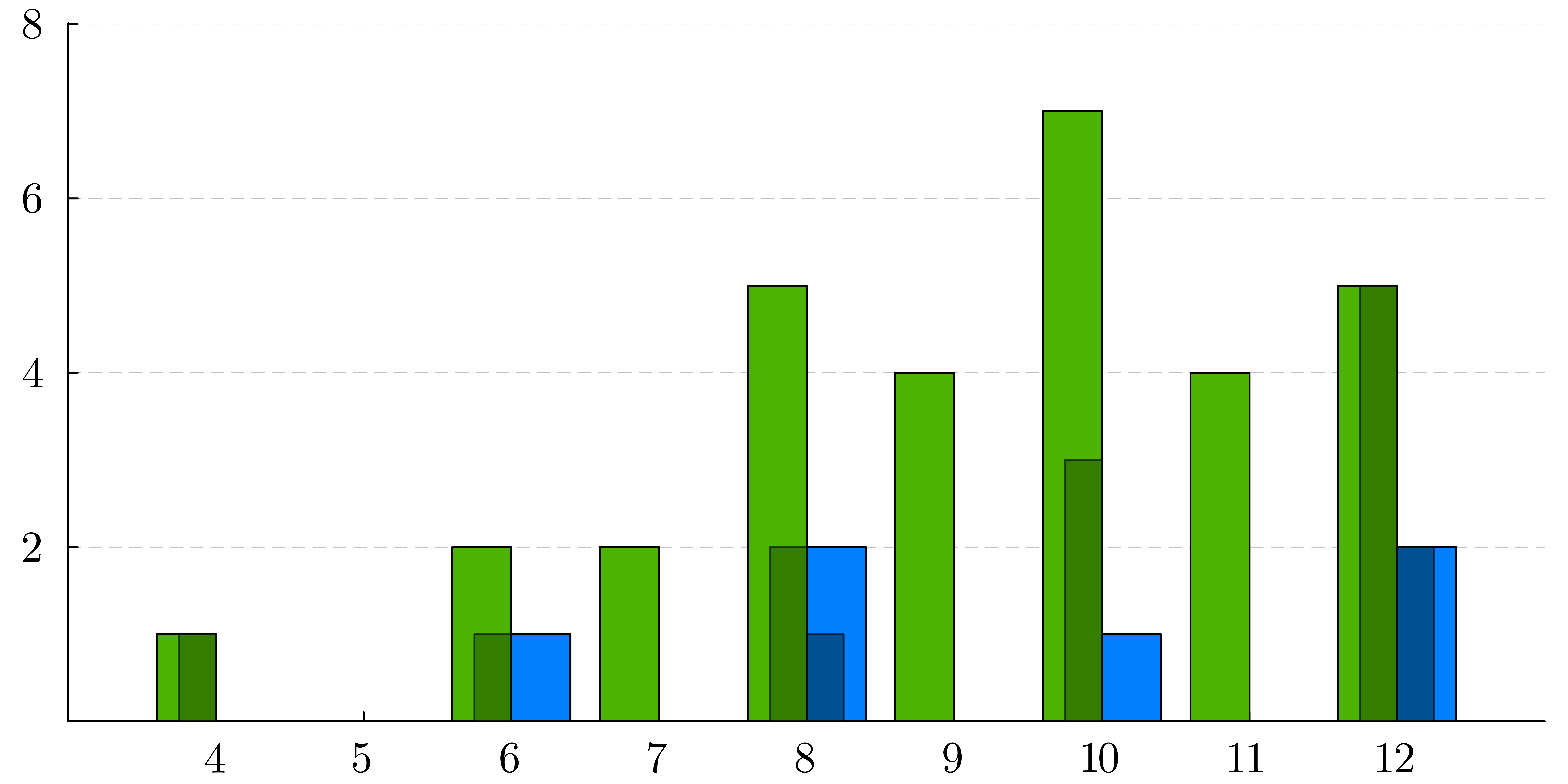}
    \caption{Number of combinatorial types of slices of $C_4$ by number of vertices.}
    \label{fig:4_cube_vertices}
\end{figure}

\begin{figure}[ht]
    \centering
    \includegraphics[width=0.48\linewidth]{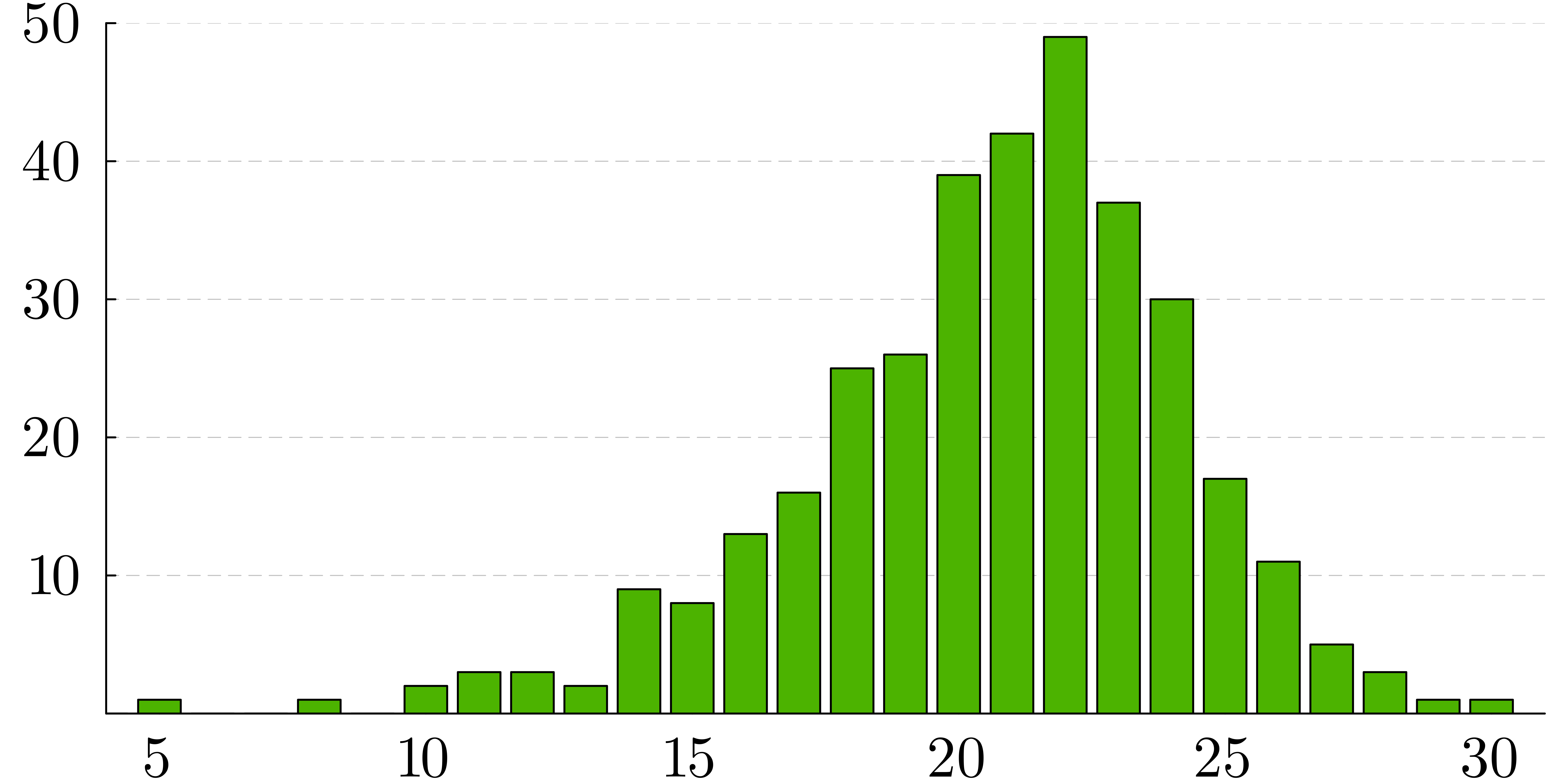}
    \:
    \includegraphics[width=0.48\linewidth]{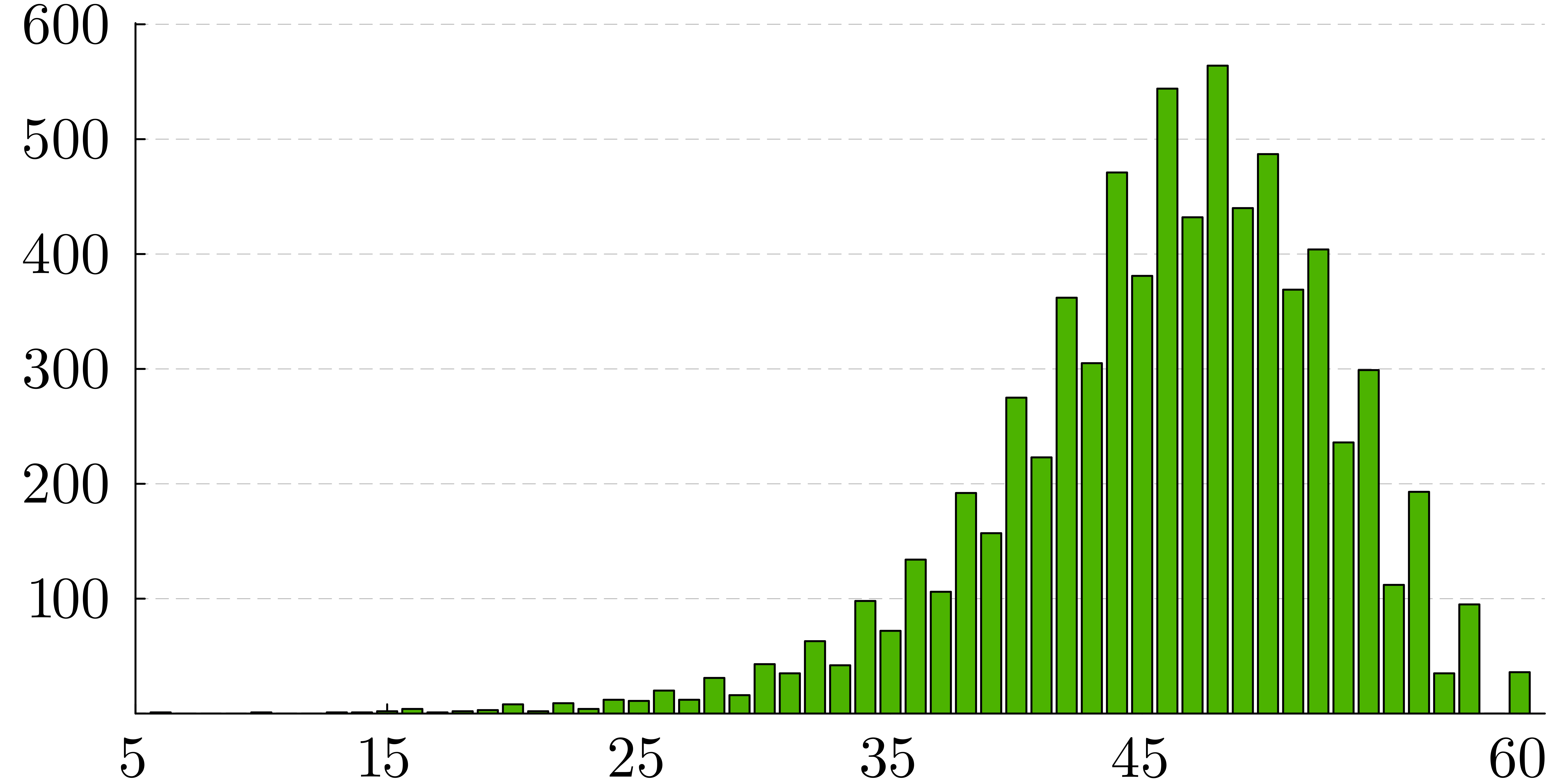}
    \caption{Number of combinatorial types of affine slices of $C_5$ and $C_6$ by number of vertices.}
    \label{fig:5_6cube_vertices}
\end{figure}

Moreover, \cite[Theorem 1.4]{deloera2025numberverticeshyperplanesection} shows that there are gaps in the possible numbers of vertices of slices of cubes. However, the authors also allow polytopes of dimension strictly lower than $d-1$, obtained as intersections of $C_d$ with a hyperplane. This relaxation yields a larger set of vertex numbers. Our computations indicate that some of these are attained \emph{only} by lower-dimensional sections. This is evident for hyperplane sections with fewer than $d+1$ vertices (by a simple dimension count), but these are not the only cases. These observations motivate the following conjecture, which we have verified computationally for $d \le 6$.
\begin{conjecture}
    There are no $(d-1)$-dimensional slices of the cube $C_d$ with $2^i$ vertices, for any $i$ satisfying $2^i\leq 2d-3$.
\end{conjecture}

Let us now focus on the maximum possible number of vertices of a slice of a cube.
O`Neil showed that this number is at most $\bigl\lceil \tfrac{d}{2} \bigr\rceil \binom{d}{\lceil \frac{d}{2} \rceil}$ \cite{ONeil1971} and constructed an explicit generic affine slice attaining this bound. Our computations, summarized in \Cref{fig:4_cube_vertices,fig:5_6cube_vertices}, and detailed further in \Cref{app:distribution}, indicate that this bound can be attained not only by generic affine slices but also by generic \emph{central} slices -- a much smaller class. Indeed, this phenomenon occurs in every dimension.

\begin{theorem}\label{thm:upper_bound_vertices}
    The upper bound $\bigl\lceil \tfrac{d}{2} \bigr\rceil \binom{d}{\lceil \frac{d}{2} \rceil}$ on the number of vertices of a $(d-1)$-dimensional slice of $C_d$ is attained by a generic central slice for every $d$.
\end{theorem}
\begin{proof}
    Every edge of the cube is parallel to a coordinate directions $e_i$ for some  $i \in [d]$. A point on such an edge can be written as $v_i(\lambda) \in \R^d$, where $(v_i(\lambda))_j \in \{-1,1\}$ for $j \in [d]\setminus \{i\}$, and $(v_i(\lambda))_i = \lambda \in [-1,1]$. 
    Given a vector $u \in \R^d$, associate to each vertex of $C_d\cap u^\perp$ of the form $v_i(\lambda)$ the sets of ``plus'' and ``minus'' indices 
    \[
    \mathfrak{p}(u)=\{j\in [d]\setminus\{i\}\mid (v_i(\lambda))_j=1,\, u_j\neq 0\},\quad
    \mathfrak{m}(u)=\{j\in [d]\setminus\{i\}\mid (v_i(\lambda))_j=-1,\, u_j\neq 0\}.
    \]
    
    First suppose that $d$ is odd and take $u=(1,\dots,1)$. The equation $\langle u,v_i(\lambda)\rangle=0$ reads $|\mathfrak{p}(u)|-|\mathfrak{m}(u)|+\lambda=0$, where $|\cdot |$ denotes the cardinality of a set, so $\lambda \in \{-1,0,1\}$. Because $\mathfrak{p}(u)\sqcup\mathfrak{m}(u)\sqcup\{i\}=[d]$, we get
    \[
        d = |\mathfrak{p}(u)| + |\mathfrak{m}(u)| + 1 = 2|\mathfrak{m}(u)| + 1 - \lambda 
        \quad\iff\quad
        2|\mathfrak{m}(u)| = d-1+\lambda \, .
    \]
    Since $d-1$ is even, it follows that $\lambda =0$ and consequently $|\mathfrak{p}(u)|=|\mathfrak{m}(u)|=\frac{d-1}{2}$. In other words, $u^\perp$ does not contain any vertex of $C_d$, and the corresponding slice is generic. 
    For each $i \in [d]$, there are $\binom{d-1}{(d-1)/2}$ choices for $\mathfrak{p}(u)$, which uniquely determines $\mathfrak{m}(u)$. The total number of vertices $v_i(\lambda)$ of the slice $C_d\cap u^\perp$ is thus
    \[
        d \binom{d-1}{\frac{d-1}{2}} 
        = \frac{d+1}{2} \binom{d}{\frac{d+1}{2}} \, .
    \]

    Now let $d$ be even and take $u=(1,\dots,1,0)$. Then, $0 =\langle u,v_i(\lambda)\rangle = |\mathfrak{p}(u)| -| \mathfrak{m}(u)| + |\{i\} \setminus \{d\}| \lambda$, and $[d] = \mathfrak{p}(u) \sqcup \mathfrak{m}(u) \sqcup \{i,d\}$. 
    If $i=d$, then $v_d(\lambda)$ is a vertex of the slice $C_d\cap u^\perp$ if and only if $|\mathfrak{p}(u)| = |\mathfrak{m}(u)|$. However, $d = |\mathfrak{p}(u)| + |\mathfrak{m}(u)| + 1 = 2|\mathfrak{p}(u)| + 1$ contradicts the assumption that $d$ is even. Thus, $u^\perp$ does not intersect edges of $C_d$ in direction $e_d$.
    If $i \neq d$, then $v_d(\lambda)$ is a vertex of the slice if and only if $|\mathfrak{p}(u)| - |\mathfrak{m}(u)| + \lambda = 0$, so $\lambda \in \{-1,0,1\}$. Because $\mathfrak{p}(u) \sqcup \mathfrak{m}(u) \sqcup \{i,d\} = [d]$, we get 
    \[
        d = |\mathfrak{p}(u)| + |\mathfrak{m}(u)| + 2 = 2|\mathfrak{m}(u)| + 2 - \lambda 
        \quad\iff\quad
        2|\mathfrak{m}(u)| = d-2+\lambda \, .
    \]
    Since $d-2$ is even, we have $\lambda =0$ and $|\mathfrak{p}(u)|=|\mathfrak{m}(u)|=\frac{d-2}{2}$. As in the previous case, this implies that $u^\perp$ does not contain any vertex of $C_d$, so the slice is generic.
    For each $i \in [d-1]$, there are $\binom{d-2}{(d-2)/2}$ choices for $\mathfrak{p}(u)$, which uniquely determines $\mathfrak{m}(u)$. The $d^{\text{th}}$ coordinate can be chosen as $+1$ or $-1$. The total number of vertices $v_i(\lambda)$ of the slice $C_d\cap u^\perp$ is thus
    \[
        2 (d-1) \binom{d-2}{\frac{d-2}{2}} 
        = \frac{d}{2} \binom{d}{\frac{d}{2}} \, .   \qedhere
    \]
\end{proof}

\Cref{thm:upper_bound_vertices} shows that for any $(d-1)$-dimensional affine subspace, the number of vertices of its intersection with the cube is maximized by restricting to intersections with generic linear subspaces. Preliminary computations suggest that an analogous phenomenon holds for intersections with $k$-dimensional subspaces.

\begin{conjecture}
    The upper bound on the number of vertices of a $k$-dimensional slice of $C_d$ can be attained by a generic central $k$-dimensional slice for every $d$ and every $k \leq d-1$.
\end{conjecture}

Observe that any centrally symmetric $k$-dimensional polytope can be realized as the intersection of some $d$-dimensional cube with a $k$-dimensional linear subspace, for a suitable $d \geq k$. Consequently, the previous conjecture would follow from the stronger statement below.

\begin{conjecture}
    For every centrally symmetric $k$-dimensional polytope $P$, the upper bound on the number of vertices of a $(k-1)$-dimensional slice of $P$ is attained by a central slice.
\end{conjecture}

We conclude this section by strengthening \cite[Lemma 2.1]{deloera2025numberverticeshyperplanesection} and showing that for any polytope, a slice with the largest possible number of vertices must be generic with respect to that polytope. This applies in particular to cubes, but we state and prove it in full generality.
\begin{proposition}\label{prop:max_vertex_slice_generic}
    Let $P\subset \R^d$ be a full-dimensional polytope with $d\geq 3$. If $P \cap H$ is a $(d-1)$-dimensional affine slice having the largest possible number of vertices among all such slices of $P$, then $H$ does not contain any vertex of $P$. 
\end{proposition}
\begin{proof}
    The strategy is to show that any slice $P\cap H$ containing a certain number of vertices of $P$ can be perturbed slightly so that $H$ contains exactly one vertex of $P$, without diminishing the number of vertices of the slice. If we then prove that, for a slice containing exactly one vertex of \(P\), there exists another slice with strictly more vertices, the claim follows. We start from this latter statement for simplicity.

    Assume that the affine hyperplane $H$ contains exactly one vertex of $P$. Without loss of generality, this vertex is the origin and $H = u^\perp$ for some $u\in\R^d$. Assume that the edges of $P$ incident to the origin are 
    \[
    [0, v_i] \quad \text{for } i = 1,\ldots,k
    \]
    for some $k \in \N$. Since the slice $P\cap u^\perp$ has dimension $d-1$ and contains no other vertex of $P$, the hyperplane $u^\perp$ cannot be supporting $P$. Therefore, it induces a nontrivial partition (i.e., both sets have at least one element) of the vertices $v_1,\ldots,v_k$ according to the sign of $\langle u,v_i\rangle$. Because $k\geq d \geq 3$, one part has at least two elements, and without loss of generality, let this be the set with positive inner product.
    In other words, there exists $I\subset [k]$ with $|I|\geq 2$ such that $\langle u, v_i \rangle >0$ for all $i \in I$. Choose $\varepsilon >0$ such that
    \[
    \varepsilon < \min \{ \langle u, v \rangle \mid \langle u, v \rangle >0,  v \in \vertices(P)\},
    \]
    and consider the parallel hyperplane $H_\varepsilon = \{x\in\R^d \mid \langle u, x \rangle = \varepsilon\}$. Then, the number of vertices of the associated slice satisfies 
    \[
    |\vertices(P\cap  H_\varepsilon)| = |\vertices(P\cap u^\perp)| - 1 + |I| > |\vertices(P\cap u^\perp)|.
    \]
    
    Now assume that the hyperplane $H$ contains more than one vertex of $P$. Without loss of generality, one of these vertices is the origin and $H = u^\perp$. Since $P$ is $d$-dimensional, at least one of the halfspaces defined by $u^\perp$ contains other vertices of $P$; assume this is the side $\{\langle u, x \rangle>0\}$. Let $v_1,\ldots,v_k$ be the remaining vertices of $P$ lying in $u^\perp$. Choose $u'\in \R^d$ so that $\{\langle u', v_i \rangle\geq 0\}$ for all $i = 1,\ldots,k$. Because there are vertices of $P$ in $\{\langle u, x \rangle>0\}$, each $v_i$ is the endpoint of at least one edge contained in $\{\langle u, x \rangle \geq 0\}$ and not in $\{\langle u, x \rangle=0\}$. Pick $\varepsilon >0$ such that
    \[
    \varepsilon < \min \{ \alpha >0 \mid \langle u - \alpha u', v \rangle > 0, \langle u, v \rangle >0,  v \in \vertices(P)\},
    \]
    and consider the rotated hyperplane $(u - \varepsilon u')^\perp$.
    By construction, this hyperplane intersects every edge of the form $[v_i,w]$ with $w\in \vertices(P)$ and $\langle u, w \rangle >0$. Indeed,
    \[
    0 = \langle u - \varepsilon u', (1-\lambda) v_i + \lambda w \rangle = 0 + (1-\lambda) \varepsilon \langle u', v_i \rangle + \lambda \langle u - \varepsilon u', w \rangle
    \]
    is satisfied by
    \[
    \lambda = \frac{\varepsilon\langle u', v_i \rangle}{\varepsilon\langle u', v_i \rangle + \langle u - \varepsilon u', w \rangle} \quad\in [0,1].
    \]
    Let $E_i\geq 1$ be the number of such edges of $P$ contained in $\{x \in \R^d \mid  \langle u, x \rangle \geq 0\}$ and having $v_i$ as endpoint. Then
    \[
    |\vertices(P\cap  (u - \varepsilon u')^\perp)| = |\vertices(P\cap u^\perp)| - k + \sum_{i=1}^k E_i \geq  |\vertices(P\cap u^\perp)|. \qedhere
    \]
\end{proof}

This result agrees with our cube computations: slices with the largest number of vertices must be generic. For cubes even more is true, namely that central slices with the largest number of vertices must be generic central slices. This follows directly from \Cref{prop:max_vertex_slice_generic} together with \Cref{thm:upper_bound_vertices}. Moreover, a minor adaptation of the proof of \Cref{prop:max_vertex_slice_generic} shows that when $P$ is centrally symmetric, the central slice (i.e., one containing the center of symmetry) with largest possible number of vertices must be a generic central slice.

\section{Exact and certified numerical algorithms}\label{sec:algo}

In this section we examine the algorithms used to compute the data discussed above. We describe two main routines, which we will compare, along with several preprocessing steps. These algorithms apply to general polytopes, not only to cubes, so we present them in full generality. The setting differs slightly depending on whether we consider central or affine slices. When a $d$-dimensional polytope $P$ is embedded in $\R^d$ with its center at the origin, all central slices are obtained by intersecting $P$ with hyperplanes through the origin. In contrast, to parametrize all affine slices of $P$, we embed $P$ into $\R^{d+1}$ at height $1$, namely in the hyperplane $\{x \in \R^{d+1} \mid x_{d+1} = 1\}$. In this representation, every affine slice of $P$ arises as the intersection with a hyperplane through the origin of $\R^{d+1}$.

Our approach follows naturally from \cite{Brandenburg2025}, but we include it here for completeness. Let $P$ be a $d$-dimensional polytope and let $A$ be the matrix whose columns are the vertices of $P$. We consider the hyperplane arrangement
\begin{equation}\label{eq:master_hyp_arrangement}
\mathcal{H}_A = \mathcal{H}_P = \{ v^\perp \mid v \in \operatorname{cols}(A)=\vertices(P)\}.
\end{equation}
Since $P$ is $d$-dimensional, each hyperplane section $P\cap u^\perp$ has dimension at most $(d-1)$. By \cite[Lemma 2.4]{Berlow2022}, if $u_1, u_2$ belong to the same cell of the hyperplane arrangement $\mathcal{H}_P$ in \eqref{eq:master_hyp_arrangement}, then they induce the same partition of the vertices of $P$, and hence the slices $P\cap u_1^\perp$ and $P\cap u_2^\perp$ have the same combinatorial type. Therefore, to enumerate all combinatorial types of slices of $P$, it suffices to collect one representative point per cell of $\mathcal{H}_P$ and compute the combinatorial type of the corresponding slice.

The first step -- collecting one representative point per cell -- encodes the main computational complexity of the algorithm, and we develop two different approaches for it. We refer to the first as the \emph{exact algorithm} which is outlined in \Cref{algo:exact}. This method computes the full cell decomposition of the hyperplane arrangement (see \cite{Kastner2020} for the implementation in \texttt{Polymake}). Once the decomposition is obtained, since all relevant cells of $\mathcal{H}_P$ are unbounded, for each cell it suffices to take the sum of its rays to obtain a point in its interior. 
\begin{algorithm}
    \Require{Polytope $P$}
    \Ensure{One point per cell of $\mathcal{H}_P$}
    \begin{algorithmic}[1]
        \State construct the central hyperplane arrangement $\mathcal{H}_P$  
        \State \texttt{all\_points} $\gets$ empty
        \For{cells of $\mathcal{H}_P$}
        \State compute the rays of the cell
        \State sum the rays to get a point $p$ in the relative interior of the cell
        \State append $p$ to \texttt{all\_points}
        \EndFor
        \Return \texttt{all\_points}
    \end{algorithmic}
    \caption{exact algorithm}\label{algo:exact}
\end{algorithm}

The computational complexity of \Cref{algo:exact} is polynomial in the number of vertices of $P$ for fixed dimension $d$, but grows exponentially with $d$ in general. This follows from \cite{Gritzmann1993} and is also explained in \cite[Section 4]{Brandenburg2025}. For instance, the first column of \Cref{tab:number_cells} shows the number of maximal cells of $\mathcal H_{C_d}$ for $d\leq 9$. Therefore, even when $P$ has relatively few vertices, in higher dimension (where ``high'' in this scenario means at least five or six) computing the exact decomposition of $\mathcal{H}_P$ quickly becomes infeasible.

To extend the computation further, we introduce a second method, which we refer to as the \emph{certified numerical algorithm}. The underlying theory goes back to \cite{Varchenko95:CriticalPoints,Huh2013} and has recently been generalized in \cite{Reinke2024}. For algorithmic aspects, see also \cite{Breiding2024}. For the convenience of the reader, we describe the details of the algorithm below, omitting proofs. 

Let $\ell_1,\ldots,\ell_n$ be the linear forms associated with the vertices of $P$, namely $\ell_i(x) = \langle v_i, x\rangle$, for $\vertices(P) = \{v_1,\ldots,v_n\}$. Consider the function 
\begin{equation}\label{eq:likelyhood}
    \psi(x) = \sum_{i=1}^n u_i \log|\ell_i(x)| - v \log |g(x)|,
\end{equation}
where $g$ is a generic quadratic polynomial with $g(x)>0$, and $u_i,v \in \R$ are some parameters. By \cite[Theorem 1.2]{Reinke2024}, if $v$ is sufficiently large (specifically, if $2v \geq \sum_{i=1}^n u_i$), then the number of critical points of $\psi$ equals the number of connected components of the complement of $\cup_{i=0}^n \{\ell_i(x)=0\}$. 
Moreover, each critical point of $\psi$ lies in exactly one such component. 
Computing these critical points amounts to solving a system of polynomial equations. This can be done efficiently with classical methods from numerical algebraic geometry, such as \texttt{HomotopyContinuation.jl} \cite{Breiding2018}. Despite being numerical, this approach allows one to \emph{certify} the computed solutions (see \Cref{rmk:certification}).
\begin{algorithm}
    \Require{Matrix $A$ with columns $A_1,\dots, A_n$ 
    }
    \Ensure{One point in each maximal cell of $\mathcal{H}_A$}
    \begin{algorithmic}[1]
        \State $\psi \gets \sum_{i=1}^n u_i \log|\langle A_i, x\rangle| - v \log |g(x)|$ 
        \State sols $ \gets \operatorname{solve}(\nabla_x \psi = 0)$ 
        \Return sols
    \end{algorithmic}
    \caption{sample points in maximal cells of $\mathcal H_A$}\label{alg:solve_hc}
\end{algorithm}

\Cref{alg:solve_hc} corresponds to \cite[Algorithm 1]{Reinke2024}. Strictly speaking, in order to find the zeros of the gradient one must apply monodromy methods to solve the resulting system of rational equations. Since this is not the focus of the present paper, we omit these details and refer to \cite{Reinke2024}. When $P = C_d$, the choice $g(x) = \sum_{i} x_i^2 + 1$ is sufficiently generic, see \Cref{subsec:symmetries}. 
This algorithm alone does not fully solve our problem, since it produces points only in the \emph{maximal} cells of the hyperplane arrangement $\mathcal{H}_P=\mathcal{H}_A$. To also collect points in the lower-dimensional cells, we iteratively restrict to intersections of hyperplanes in the arrangement and apply \Cref{alg:solve_hc} in that restricted setting. Concretely, to restrict to $v_1^\perp$, where $v_1$ is the first column of $A$, we define the matrix $A|_{v_1}$ as 
\begin{equation}\label{eq:restriction_matrix}
(A|_{v_1})_{i,j} = A_{i,j} - (v_1)_{i} \frac{A_{k,j}}{(v_1)_k} \qquad \hbox{ for } i \neq k, \; j = 2, \ldots, |\vertices(P)|,
\end{equation}
for any index $k$ with $(v_1)_k\neq 0$. In practice, we simply pick the first such index.
To restrict further, for example to $v_1^\perp \cap v_2^\perp$, we repeat the same procedure on the new matrix $A|_{v_1}$.

\begin{remark}
The restricted matrix $A|_\mathfrak{v}$, for a subset $\mathfrak{v} \subset \vertices(P)$, corresponds to restricting the hyperplane arrangement $\mathcal{H}_A$ to the face $\cap_{v \in \mathfrak{v}} v^\perp$ whose codimension equals $\dim(\spn \mathfrak{v})$. From the perspective of slices, points in this face of the arrangement represent normals to slices of the cube that contain all vertices in $\mathfrak{v}$.
\end{remark}

\begin{figure}[ht]
    \centering
    \includegraphics[width=0.35\linewidth]{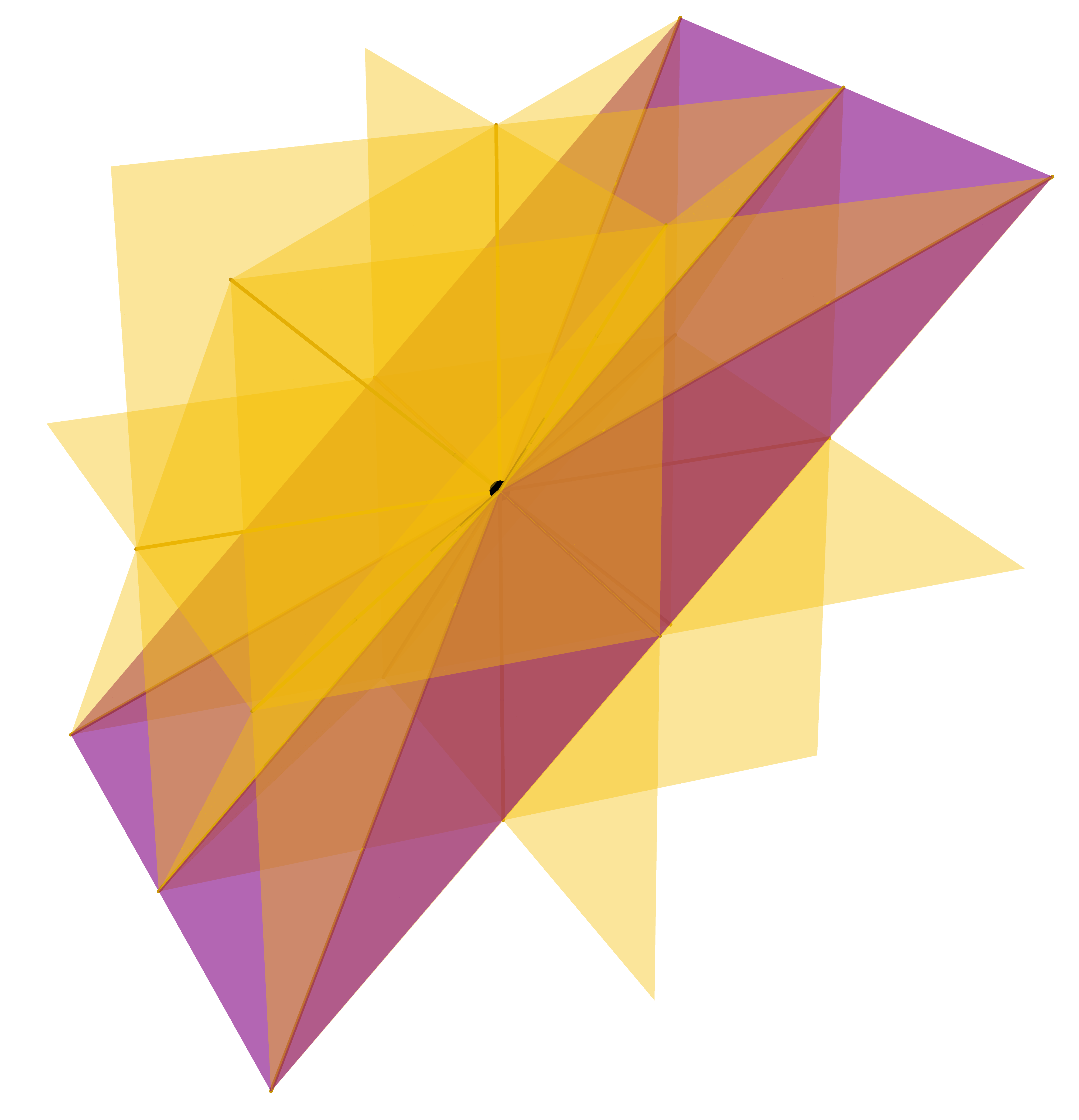}
    \qquad \quad \qquad
    \includegraphics[width=0.33\linewidth]{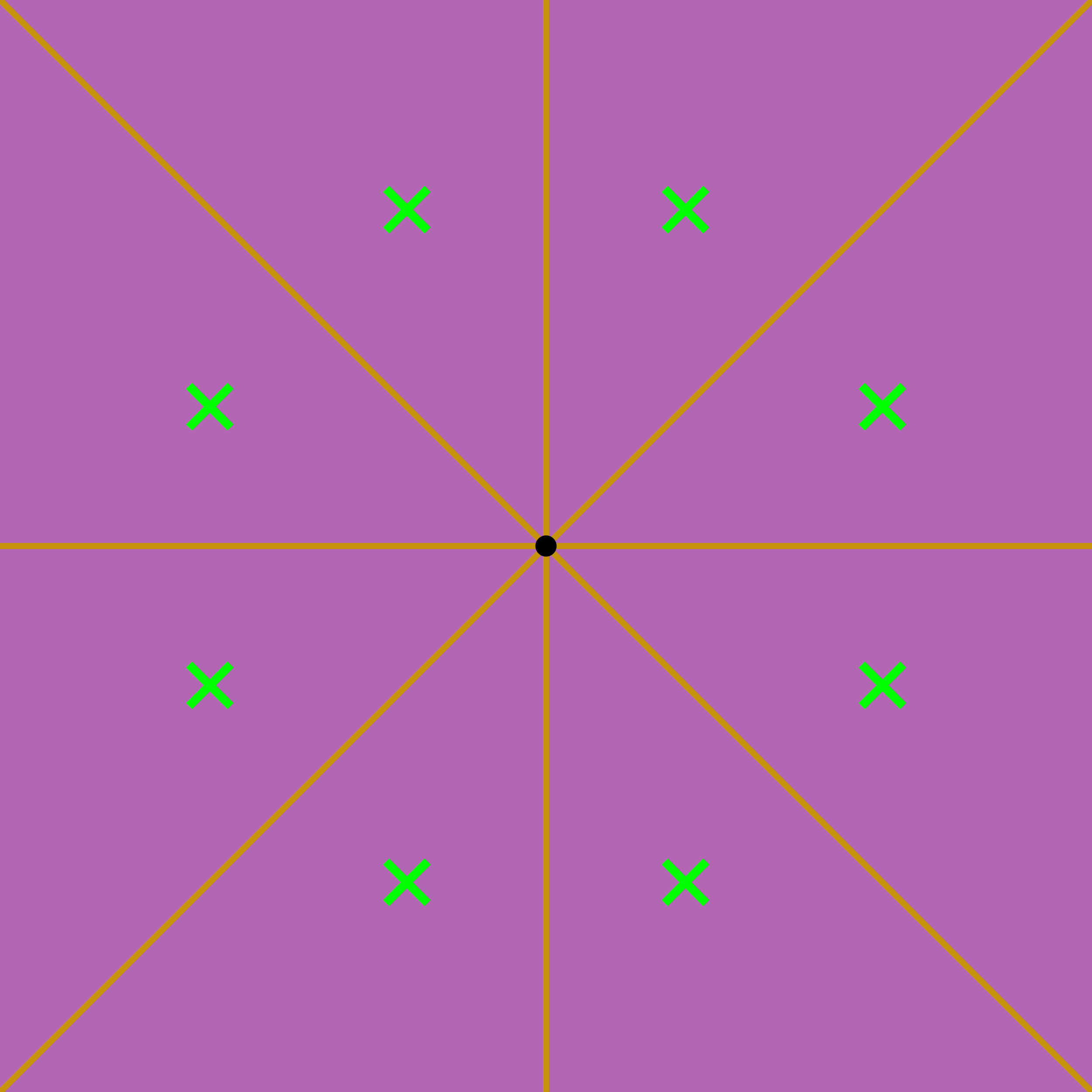}
    \caption{Hyperplane arrangements $\mathcal{H}_{A|_{v_1}}$ (left), $\mathcal{H}_{A|_\mathfrak{v}}$ (right) from \Cref{ex:restriction_matrices}.}
    \label{fig:redueced_hyp_arr}
\end{figure}
\begin{example}[Restriction to subspaces]\label{ex:restriction_matrices}
    Let $P=C_3\subset \R^4$ be embedded at height $x_4 = 1$, and consider the associated matrix
    \[
    A = \begin{blockarray}{ccccccccc}
        {\color{gray} \text{\footnotesize $v_1$}} & {\color{gray} \text{\footnotesize $v_2$}} & {\color{gray} \text{\footnotesize $v_3$}} & {\color{gray} \text{\footnotesize $v_4$}} & {\color{gray} \text{\footnotesize $v_5$}} & {\color{gray} \text{\footnotesize $v_6$}} & {\color{gray} \text{\footnotesize $v_7$}} & {\color{gray} \text{\footnotesize $v_8$}}\\
        \begin{block}{(cccccccc)c}
        -1 & 1 & -1 & 1 & -1 & 1 & -1 & 1 & {\color{gray} \text{\footnotesize 1}}\\
        -1 & -1 & 1 & 1 & -1 & -1 & 1 & 1 & {\color{gray} \text{\footnotesize 2}}\\
        -1 & -1 & -1 & -1 & 1 & 1 & 1 & 1 & {\color{gray} \text{\footnotesize 3}}\\
        1 & 1 & 1 & 1 & 1 & 1 & 1 & 1 & {\color{gray} \text{\footnotesize 4}}\\
        \end{block}
    \end{blockarray}
    \]
    with the vertices of $C_3$ as columns. For $v_1 = (-1,-1,-1,1)$ the restriction matrix is
    \begin{equation}\label{eq:restricted_matrix_v1}
    A|_{v_1} = \begin{blockarray}{cccccccc}
        {\color{gray} \text{\footnotesize $v_2$}} & {\color{gray} \text{\footnotesize $v_3$}} & {\color{gray} \text{\footnotesize $v_4$}} & {\color{gray} \text{\footnotesize $v_5$}} & {\color{gray} \text{\footnotesize $v_6$}} & {\color{gray} \text{\footnotesize $v_7$}} & {\color{gray} \text{\footnotesize $v_8$}}\\
        \begin{block}{(ccccccc)c}
        -2 & 2 & 0 & 0 & -2 & 2 & 0 & {\color{gray} \text{\footnotesize 2}}\\
        -2 & 0 & -2 & 2 & 0 & 2 & 0 & {\color{gray} \text{\footnotesize 3}}\\
        2 & 0 & 2 & 0 & 2 & 0 & 2 & {\color{gray} \text{\footnotesize 4}}\\
        \end{block}
    \end{blockarray}.
    \end{equation}
    The associated hyperplane arrangement $\mathcal{H}_{A|_{v_1}} \subset \R^3 = v_1^\perp$ is shown in \Cref{fig:redueced_hyp_arr}, left.
    Suppose now we are interested in the tuple $\mathfrak{v} = \{v_1, v_4\}$. In this case, we need to iterate the restriction procedure. The hyperplane of $\mathcal{H}_{A|_{v_1}}$ corresponding to $v_4$ is highlighted in purple. Note that we cannot apply formula \eqref{eq:restriction_matrix} with $k=1$, since this would involve division by zero in the transformed vector $v_4$, namely the third column of $A|_{v_1}$. Instead, choosing $k=2$ yields
    \begin{equation}\label{eq:restricted_matrix_v1v4}
    A|_\mathfrak{v} = \left( (A|_{v_1})_{i,j} - (A|_{v_1})_{i,3} \frac{(A|_{v_1})_{2,j}}{(A|_{v_1})_{2,3}} \right)_{i,j} = \begin{blockarray}{ccccccc}
        {\color{gray} \text{\footnotesize $v_2$}} & {\color{gray} \text{\footnotesize $v_3$}} & {\color{gray} \text{\footnotesize $v_5$}} & {\color{gray} \text{\footnotesize $v_6$}} & {\color{gray} \text{\footnotesize $v_7$}} & {\color{gray} \text{\footnotesize $v_8$}}\\
        \begin{block}{(cccccc)c}
        -2 & 2 & 0 & -2 & 2 & 0 & {\color{gray} \text{\footnotesize 2}}\\
        0 & 0 & 2 & 2 & 2 & 2 & {\color{gray} \text{\footnotesize 4}} \\
        \end{block}
    \end{blockarray}.
    \end{equation}
    The hyperplane arrangement $\mathcal{H}_{A|_\mathfrak{v}} \subset \R^2$, lying in the purple two-dimensional plane $v_1^\perp\cap v_4^\perp$, is displayed in \Cref{fig:redueced_hyp_arr}, right.
\end{example}

\renewcommand{\arraystretch}{1.4}
\begin{table}[ht!]
\hspace*{-1em}
    \centering
    \begin{tabular}{c|ccccccc}
    $d$ & 0 & 1 & 2 & 3 & 4 & 5 & 6 \\
    \hline
    $3$
    & \footnotesize 104 & \footnotesize 32 & \footnotesize 6, 8 & \footnotesize 2 & & & \\
    \hline
    $4$ 
    & \footnotesize 1882 & \footnotesize 370 & \footnotesize 32, 60 & \footnotesize 6, 8, 10 & \footnotesize 2 & & \\
    \hline
    \multirow{2}{*}{$5$}
    & \multirow{2}{*}{\footnotesize 94572} & \multirow{2}{*}{\footnotesize 11292} & \footnotesize 370, 1024 & \footnotesize 32, 60, 96, & \footnotesize 6, 8, 10, & \multirow{2}{*}{\footnotesize 2} & \\[-4pt]
     &  &  & \footnotesize 1296 & \footnotesize 98, 128 & \footnotesize 12, 14, 16 &  & \\
    \hline
    \multirow{4}{*}{$6$}
    & \multirow{4}{*}{\footnotesize 15028134} & \multirow{4}{*}{\footnotesize 1066044} & \footnotesize 11292,  & \footnotesize 370, 1024, & \footnotesize 32, 60, 96, 98, & \footnotesize 6, 8, 10, & \multirow{4}{*}{\footnotesize 2} \\[-4pt]
    &  &  & \footnotesize 47900, & \footnotesize 1296, 2258, & \footnotesize 128, 144, 146, 180, & \footnotesize 12, 14, 16, &  \\[-4pt]
    &  &  & \footnotesize 47900, & \footnotesize 2640, 3790, & \footnotesize 200, 220, 264, & \footnotesize 18, 20, &  \\[-4pt]
    &  &  & \footnotesize 73632 & \footnotesize 5040 & \footnotesize 288, 312, 336 & \footnotesize 22, 24 & \\
    \hline $7$ & \tiny 8378070864 & \tiny 347326352 &  &  &  &  & \\
    \hline $8$ & \tiny 17561539552946 & \tiny 419172756930 &  &  &  &  & \\
    \hline $9$ & \tiny 144130531453121108 & \tiny 1955230985997140 &  &  &  &  & \\
    \end{tabular}
    \caption{Possible numbers of maximal cells in the hyperplane arrangement $\mathcal{H}_{A|_\mathfrak{v}}$ in $\R^{d+1-\dim(\spn\mathfrak{v})}$ for $P = C_d \times \{1\} \subset\R^{d+1}$ or $P = C_{d+1}\subset\R^{d+1}$. The first column of the table fixes $d$, the other columns denote $\dim(\spn \mathfrak{v})$.}
    \label{tab:number_cells}
\end{table}
We collect in \Cref{tab:number_cells} the numbers of maximal cells in the hyperplane arrangements $\mathcal{H}_{A|_\mathfrak{v}}$ for all subsets $\mathfrak{v}\subset\vertices(C_d)$, for $d\leq 6$, in the setting of affine slices, namely for $C_d\subset\R^{d+1}$. As explained in \Cref{sec:generic-central-slices}, the maximal cells of $\mathcal H_A$ are in bijection with threshold functions on $d$ (or fewer) variables, and the first column of \Cref{tab:number_cells} reports the corresponding numbers, according to the OEIS sequence \href{https://oeis.org/A000609}{A000609}.
The second column displays the \emph{resonance arrangement} from \cite[Section 6.3]{Brysiewicz2023}, which coincides with OEIS sequence \href{https://oeis.org/A034997}{A034997}. This can be seen by considering the cube $[0,1]^d$ instead of $[-1,1]^d$, and taking $\mathfrak{v} = \{(0,\ldots,0)\}$.

Once we have the restricted matrix $A|_\mathfrak{v}$, we compute the critical points of the associated function $\psi$ using \Cref{alg:solve_hc}. This yields a list of points in $\R^m$, lying inside the linear space $\cap_{v\in \mathfrak{v}} v^\perp$, where $m+\dim(\spn \mathfrak{v})$ is the dimension of the ambient space. To express these points in the original coordinates of the ambient space, we must invert the restriction procedure. 

Consider $A|_{\mathfrak{v}}$ and $A|_{\mathfrak{v}\cup v_\ell}$, with $\dim(\spn (\mathfrak{v} \cup v_\ell)) = \dim(\spn \mathfrak{v}) +1$, and let $\widetilde{v}_\ell$ be the column of $A|_{\mathfrak{v}}$ corresponding to $v_\ell$. 
Let $k$ be the index of the row removed when going from $A|_{\mathfrak{v}}$ to $A|_{\mathfrak{v}\cup v_\ell}$, and let $p\in \R^{m}$ be a point in $(\cap_{v\in \mathfrak{v}} v^\perp) \cap v_\ell$. We map $p$ to the point $\hat{p} = (p_{i_1},\ldots,p_k, \ldots,p_{i_m})$, $i_1 < \ldots < k < \ldots < i_m$, such that 
\[
\langle (A|_{\mathfrak{v}})_{v_\ell}, \hat{p}\rangle = 0.
\]
We repeat this step iteratively to recover the missing coordinates, until we get back to ambient space.
All of the above steps are summarized in \Cref{alg:numerical}.

\begin{algorithm}
    \Require{Polytope $P$ 
    }
    \Ensure{One point per cell of $\mathcal{H}_P$}
    \begin{algorithmic}[1]
        \State \texttt{all\_points} $\gets$ empty
        \State $A \gets$ matrix with columns the vertices of $P$
        \State \texttt{sols} $\gets$ \Cref{alg:solve_hc}$(A)$ 
        \State append \texttt{sols} to \texttt{all\_points}
        \For{$\mathfrak{v}$ tuple of vertices of $P$}\label{algo_step:vert_tuples}
        \State construct the matrix $A|_\mathfrak{v}$
        \State\label{step:solve_lower} \texttt{sols\_lower} $\gets$ \Cref{alg:solve_hc}$(A|_\mathfrak{v})$, the solutions in $\R^{d+1-\dim(\spn\mathfrak{v})}$ 
        \State\label{step:back_to_space} \texttt{sols\_}$\mathfrak{v} \gets$ embedding of \texttt{sols\_lower}$\subset \cap_{v \in \mathfrak{v}} v^\perp$ into $\R^{d+1}$
        \State append \texttt{sols\_}$\mathfrak{v}$ to \texttt{all\_points}
        \EndFor
        \Return \texttt{all\_points}
    \end{algorithmic}
    \caption{certified numerical algorithm}\label{alg:numerical}
\end{algorithm}

We illustrate steps \ref{step:solve_lower} and \ref{step:back_to_space} of \Cref{alg:numerical} in a low-dimensional setting.
\begin{example}[Steps \ref{step:solve_lower} and \ref{step:back_to_space} of \Cref{alg:numerical}]
    Let $P=C_3\subset \R^4$, let $\mathfrak{v} = \{v_1, v_4\} = \{(-1,-1,-1,1), (1,1,-1,1)\}$ as in \Cref{ex:restriction_matrices}, and consider the restricted matrix $A|_\mathfrak{v}$ from \eqref{eq:restricted_matrix_v1v4}. In step \ref{step:solve_lower} of \Cref{alg:numerical}, we apply \Cref{alg:solve_hc} to compute the critical points of 
    \[
    \psi(x_2,x_4) = u_1 \log |2x_2| + u_2 \log|2 x_4| + u_3 \log|2x_4-2x_2| + u_4 \log|2x_2+2x_4| + v \log |x_2^2+x_4^2+1|.
    \]
    In practice, before constructing $\psi$, we discard redundant columns of $A|_\mathfrak{v}$ (up to sign). Here, we keep only one among $(2,0)$ and $(-2,0)$, and one among the two copies of $(0,2)$.
    We then solve the polynomial system
    \begin{gather*}
        \frac{u_1}{x_2} - \frac{2u_3}{-2x_2 + 2x_4} + \frac{2u_4}{2x_2 + 2x_4} - \frac{2vx_2}{1 + x_2^2 + x_4^2} = 0 \\
        \frac{u_2}{x_4} + \frac{2u_3}{-2x_2 + 2x_4} + \frac{2u_4}{2x_2 + 2x_4} - \frac{2vx_4}{1 + x_2^2 + x_4^2} = 0
    \end{gather*}  
    with $(u,v) = (1,1,1,1,4)$, obtaining in about $0.06$ seconds the eight points
    \[
    (\pm 0.38268, \pm 0.92387),\; ( \pm 0.92387, \pm 0.38268),
    \]
    shown as green diagonal crosses in \Cref{fig:redueced_hyp_arr}, right. 
    In step \ref{step:back_to_space}, we embed these points back into $\R^4$. This is done by reintroducing the coordinates one at a time, in reverse order of restriction. First, the restricted $v_4$ (the third column of $A|_{v_1}$ in \eqref{eq:restricted_matrix_v1}) enforces $-x_3+x_4=0$.
    Next, $v_1$ imposes $-x_1-x_2-x_3+x_4=0$. Altogether, $\mathfrak{v}=\{v_1,v_4\}$ contributes to \texttt{all\_points} the eight points
    \[
    (\pm (0.38268, -0.38268), \pm (0.92387, 0.92387)),\; ( \pm (0.92387, -0.92387), \pm (0.38268, 0.38268)).
    \]
    This number of solutions matches one of the cases for $d=3$ in \Cref{tab:number_cells} , corresponding to slices through two linearly independent vertices of $C_3$.
\end{example}

\begin{remark}\label{rmk:certification}
    Although \Cref{alg:numerical} is based on numerical methods, its output can be certified from two points of view. First, \texttt{HomotopyContinuation.jl} uses interval arithmetic or Smale's $\alpha$-theory to prove that solutions of the polynomial system are real and pairwise distinct; see, e.g., \cite{bates2024} for a recent survey. Second, \texttt{CountingChambers.jl} allows exact computation of the number of maximal open cells in a hyperplane arrangement, which equals the number of solutions to the polynomial systems we consider. Together, these guarantees certify that all solutions are found, and that they are both real and distinct.\\
    Our computations are not only certified but also numerically stable. The floating-point solutions to our systems of equations represent normals to hyperplane slices. Each such point lies in the interior of a full-dimensional cell (within an appropriate linear subspace) of the hyperplane arrangement, and within a cell the combinatorial type of the slice is constant. Thus, numerical imprecision does not affect the outcome: any small perturbation yields the same combinatorial type. Moreover, when we restrict to a lower-dimensional subspace, the resulting solutions are still floating-point approximations of normals, but the recovery of their full-dimensional coordinates is carried out by exact linear algebra. Hence, the reconstruction step introduces no additional error, and the overall computation remains numerically stable.
\end{remark}

Finally, once we have collected one representative point from each cell of the hyperplane arrangement $\mathcal{H}_P$, we construct the corresponding slices of the polytope $P$ and compare their combinatorial types. For this last step, we use the \texttt{Polymake} function \texttt{isomorphic}  \cite{GavJos:Polymake}. 
\begin{remark}
    To make the comparison efficient, we first check the $f$-vector of each new slice against those already obtained. If the $f$-vector is new, we add the slice to the list. Otherwise, we test whether the slice is isomorphic to any of the existing ones with the same $f$-vector. The command \texttt{isomorphic} performs this test by reducing it to checking graph isomorphism of the vertex-facet incidence graphs of the two slices.
\end{remark}

We summarize the discussion of this section in \Cref{alg:all}. In the step where one needs to collect one representative point per cell of the hyperplane arrangement, there is the option to either use the exact (\Cref{algo:exact}) or the certified numerical (\Cref{alg:numerical}). In order to stress that the numerical algorithm gives correct, certified output, we compare the two options where possible (see \Cref{tab:affine_slices,tab:central_slices} below).

\begin{algorithm}
    \Require{Polytope $P$ 
    }
    \Ensure{All combinatorial types of slices of $P$}
    \begin{algorithmic}[1]
        \State \texttt{slices} $\gets$ empty
        \State \texttt{all\_points} $\gets$ either \Cref{algo:exact}$(P)$ or \Cref{alg:numerical}$(P)$ 
        \For{$u \in $ \texttt{all\_points}}
        \State construct $P\cap u^\perp$
        \If{the $f$-vector of $P\cap u^\perp$ is not among the $f$-vector of any $s \in $ \texttt{slices}}
            \State append $P\cap u^\perp$ to \texttt{slices}
        \Else{ the $f$-vector of $P\cap u^\perp$ coincides with the $f$-vector of the slices in $S \subset $ \texttt{slices}}
            \If{\texttt{isomorphic}$(P\cap u^\perp, s)=$ true for any $s\in S$} 
            \State move on
            \Else{ append $P\cap u^\perp$ to \texttt{slices}}
            \EndIf
        \EndIf
        \EndFor
        \Return \texttt{slices}
    \end{algorithmic}
    \caption{all combinatorial types of slices}\label{alg:all}
\end{algorithm}

\subsection{Symmetries}\label{subsec:symmetries}
So far, the algorithm works for any polytope. When $P=C_d$, however, the high degree of symmetry can be exploited to significantly reduce the computational complexity. While a similar reduction is possible for any polytope with a nontrivial symmetry group, we restrict the discussion to the cube for clarity.

Given a hyperplane $H = \{x \in \R^d \mid  \langle u,x\rangle = t\}$, the slice $C_d\cap H$ is combinatorially equivalent to $C_d\cap \pi(H)$, where $\pi(H) = \{x \in \R^d \mid  \langle \pi(u),x\rangle = t\}$ for any signed permutation $\pi \in B_d$. This follows from the fact that $B_d$ is the full symmetry group of the cube. Hence, to accelerate the computation it is natural to exploit these symmetries  
and sample points from fewer cells. In principle, this can be incorporated directly into the certified numerical algorithm, \Cref{alg:numerical}. Indeed, the command \texttt{HomotopyContinuation.solve} accepts an optional group input describing symmetries of the solution set. However, there is a subtle mismatch between geometry and computation. Once we add the auxiliary function $g$ in the definition of $\psi$ in \eqref{eq:likelyhood}, we may break some of the symmetries of the solution set. For the cube, this issue can be avoided by choosing
\[
g(x) = \sum_{i} x_i^2 +1,
\]
which is generic enough for our purposes and invariant under the action of $B_d$. A priori, one could attempt a similar approach when restricting to lower-dimensional cells of the arrangement. In that case, however, the relevant symmetry group changes, and it would be necessary to select a different function $g$ that is both generic and symmetry-preserving.
Since this lies beyond the scope of the present work, in our computations we simply use the involution $x\mapsto -x$ together with the above choice of $g$, which remain valid under restriction to any $\mathfrak{v} \subset \vertices(C_d)$.

Another optional argument of \texttt{HomotopyContinuation.solve} that speeds up the computation is the ``target solutions count''. This option provides a stopping criterion by specifying the exact number of solutions that need to be found. In our setting, this information is given in \Cref{tab:number_cells}, whose entries were obtained using \texttt{CountingChambers.jl}.

A further way to exploit cube symmetries is in the reduction of subsets $\mathfrak{v} \subset \vertices(C_d)$. Indeed, if $\mathfrak{v} = \pi(\mathfrak{v}')$ for some signed permutation $\pi\in B_d$, then the combinatorial types of the slices corresponding to hyperplanes through $\mathfrak{v}$ coincide with those for hyperplanes through $\mathfrak{v}'$. Here if $\mathfrak{v} = \{v_1,\ldots,v_k\}$, we write $\pi(\mathfrak{v}) = \{\pi(v_1), \ldots,\pi(v_k) \}$. The proof of this claim is given in the first paragraph of the proof of \Cref{thm:generic_central_slices}. Therefore, in step \ref{algo_step:vert_tuples} of \Cref{alg:numerical}, instead of iterating over all subsets of vertices of $C_d$, we may restrict to one representative tuple in each orbit
\[
B_d \cdot \mathfrak{v} = \{ \pi(\mathfrak{v}) \mid  \pi \in B_d\},
\]
for sets $\mathfrak{v}$ consisting of at most $d$ linearly independent vertices.
With all these precautions in place, we are now ready to establish our main theorems.
\begin{proof}[Proof of \Cref{thm:main_affine}]
    Applying \Cref{alg:all} to the cube $C_d = [-1,1]^d \times\{1\} \subset \R^{d+1}$ for $d = 3$, $4$, $5$, $6$, and using the symmetry reductions from \Cref{subsec:symmetries}, we obtain exactly 
    $4$, $30$, $344$, and $7346$ distinct combinatorial types of affine slices, respectively.
\end{proof}

\begin{proof}[Proof of \Cref{thm:main_central}]
    Applying \Cref{alg:all} to the cube $C_d = [-1,1]^d \subset \R^d $ for $d = 3$, $4$, $5$, $6$, $7$, again with the symmetry reductions from \Cref{subsec:symmetries}, yields precisely 
    $2$, $6$, $23$, $133$, and $1657$ combinatorial types of central slices, respectively.
\end{proof}

\subsection{Refined data}
In this section we highlight some further details of the slices already described in \Cref{sec:description_slices}, focusing on how the computation was carried out in practice. 
When implementing \Cref{alg:all}, we examine the cells of $\mathcal{H}_P$ dimension by dimension. Concretely, we begin with the top-dimensional cells of $\mathcal{H}_P$, which correspond to slices that do not intersect any vertex of $C_d$. Next, we consider cells of codimension one, which correspond to slices through exactly one vertex of $C_d$. After $k$ iterations, we are examining cells of codimension $k$, that is, slices through $k$ linearly independent vertices of $C_d$. This process continues until $k=d$. At each step, we compare the resulting set of combinatorial types 
with those found in the previous steps, and record how many genuinely new types arise when the slice is required to pass through one additional vertex of the cube.

We collect all this information in \Cref{tab:affine_slices}, which reports the results for cubes up to dimension $d=6$. An analogous summary for central slices is given in \Cref{tab:central_slices}. In each table, the numbers in brackets in the column indexed by $k$ denote the combinatorial types that arise for the first time when the slice passes through $k$ linearly independent vertices, that is, types that cannot already be obtained from slices through fewer than $k$ vertices.

\renewcommand{\arraystretch}{1.4}
\begin{table}[ht!]
    \centering
    \begin{tabular}{c|ccccccc|c}
     & 0 & 1 & 2 & 3 & 4 & 5 & 6 & tot \\
    \hline
    \multirow{2}{*}{$d=3$} 
    & 3 & 3 & 2 & 2 & & & & \multirow{2}{*}{4} \\
    &     & (1) & (0) & (0) & & & & \\
    \hline
    \multirow{2}{*}{$d=4$} 
    & 12 & 14 & 14 & 10 & 6 & & & \multirow{2}{*}{30} \\
    &     & (7) & (6) & (4) & (1) & & & \\
    \hline
    \multirow{2}{*}{$d=5$} 
    & \textcolor{figgreen}{58} & 103 & 129 & 105 & 52 & 14 & & \multirow{2}{*}{\textcolor{figgreen}{344}}\\
    &     & (81) & (96) & (73) & (31) & (5) & & \\
    \hline
    \multirow{2}{*}{$d=6$} 
    & \textcolor{figgreen}{554} & \textcolor{figgreen}{1482} & 2296 & 2179 & 1276 & 422 & 62 & \multirow{2}{*}{\textcolor{figgreen}{7346}} \\
    &      & \textcolor{figgreen}{(1376)} & (2078) & (1917) & (1066) & (319) & (36) & \\
    \hline
    \end{tabular}
    \caption{Number of combinatorial types of slices of $C_d$ in dimensions $d=3,4,5,6$. Slices are subdivided in distinct columns depending on how many linearly independent vertices of $C_d$ they contain. \textcolor{figgreen}{Green} values require the use of the numerical one. In all other cases, the exact and certified numerical outputs coincide. Numbers in brackets stand for \emph{new} combinatorial types, that appear for the first time in that given column.}
    \label{tab:affine_slices}
\end{table}

\renewcommand{\arraystretch}{1.4}
\begin{table}[ht!]
    \centering
    \begin{tabular}{c|ccccccc|c}
     & 0 & 1 & 2 & 3 & 4 & 5 & 6 & tot \\
    \hline
    \multirow{2}{*}{$d=3$} 
    & 2 & 1 & 1 & & & & & \multirow{2}{*}{2} \\
    & & (0) & (0) & & & & & \\
    \hline
    \multirow{2}{*}{$d=4$} 
    & 3 & 2 & 2 & 2 & & & & \multirow{2}{*}{6} \\
    & & (1) & (1) & (1) & & & & \\
    \hline
    \multirow{2}{*}{$d=5$} 
    & 7 & 6 & 6 & 5 & 3 & & & \multirow{2}{*}{23}\\
    & & (5) & (5) & (4) & (2) & & & \\
    \hline
    \multirow{2}{*}{$d=6$} 
    & \textcolor{figgreen}{21} & 28 & 34 & 30 & 18 & 7 & & \multirow{2}{*}{\textcolor{figgreen}{133}} \\
    & & (27) & (33) & (29) & (17) & (6) & & \\
    \hline
    \multirow{2}{*}{$d=7$} 
    & \textcolor{figgreen}{135} & \textcolor{figgreen}{288} & 427 & 419 & 268 & 105 & 21 & \multirow{2}{*}{\textcolor{figgreen}{1657}} \\
    &  & \textcolor{figgreen}{(287)} & (426) & (418) & (267) & (104) & (20) & \\
    \hline
    \end{tabular}
    \caption{Number of combinatorial types of \emph{central} slices of the cubes in dimensions $d=3,4,5,6,7$. We refer to \Cref{tab:affine_slices} for the meaning of columns, color-coding and numbers in brackets.}
    \label{tab:central_slices}
\end{table}

Thanks to our computation, we observe a very special behavior of central slices of the cube, which leads us to the following conjecture, verified computationally up to dimension $d=7$.
\begin{conjecture}\label{conj:central_slices_only_cube}
    Let $C_d\subset \R^d$ be the $d$-dimensional cube. For each $k\in [0,d]$, let $S_k$ denote the set of all combinatorial types of central slices of $C_d$ through exactly $k$ linearly independent vertices of the cube. Then, for all integers $k_1, k_2 \in [0,d]$ the intersection $S_{k_1}\cap S_{k_2}$ consists of exactly one combinatorial type, namely the type of $C_{d-1}$.
\end{conjecture}
We stress that in the definition of $S_k$ we allow a slice to contain more than $k$ vertices of $C_d$, but these vertices must span, together with the origin, a $k$-dimensional linear space.
The following result shows that the intersection $S_{k_1}\cap S_{k_2}$ considered in \Cref{conj:central_slices_only_cube} always contains at least the combinatorial type of $C_{d-1}$. What remains open -- and would complete the proof of the conjecture --
is to show that no other combinatorial type can appear in two distinct sets $S_{k_i}$.
\begin{proposition}
    Let $C_d\subset \R^d$ be the $d$-dimensional cube. For every $k\in [0,d]$ there exists a central slice of $C_d$ through exactly $k$ linearly independent vertices of the cube which is combinatorially equivalent to $C_{d-1}$.
\end{proposition}
\begin{proof}
    If $k=0$, any slice of of the form $\{x_i=0\}$ is a cube of dimension $d-1$. If $k=d$, then any slice of the form $\{x_i = \pm x_j\}$
    is again a cube of dimension $d-1$. Now suppose $1\leq k\leq d-1$ and consider the vector
    \[
    u_k = (-(d-k),\underbrace{1,\ldots,1}_{d-k},\underbrace{0,\ldots,0}_{k-1}).
    \]
    The hyperplane $u_k^\perp$ intersects $C_d$ in exactly $k$ linearly independent vertices, namely the vertices of the $(k-1)$-dimensional faces defined by $\{ x_1 = \ldots = x_{d-k+1} = \pm 1 \}$.
    Moreover, $u_k^\perp$ intersects precisely those edges of the form $(\lambda,\pm1,\ldots,\pm1)$ in exactly one value of $\lambda$.
    All remaining edges are either fully contained in $u_k^\perp$ or do not intersect it at all. 
    Define the projection $\varphi : \R^{d} \to \R^{d-1}$ which forgets the first coordinate. Then, since $u_k \neq (\pm 1, 0,\ldots,0)$, the map $\varphi\vert_{u_k^\perp}$ is an affine isomorphism, which restricts to an affine isomorphism on $C_d \cap u_k^\perp$, whose image is $C_{d-1}$. Therefore, their face lattices are isomorphic, hence the slice $C_d \cap u_k^\perp$ is combinatorially equivalent to $C_{d-1}$.
\end{proof}

\begin{table}[!h]
    \centering
    \begin{tabular}{c|c|c|c|c}
        $0$ & $1$ & $2$ & $3$ & $4$ \\
        \hline
        $[4, 6, 4]$ & $[4, 6, 4]$ & $[4, 6, 4]$ & $[4, 6, 4]$ & $[4, 6, 4]$ \\
        $[6, 9, 5]$ & $[6, 9, 5]$ & $[6, 9, 5]$ & $[6, 9, 5]$ & $[6, 9, 5]$ \\
        \central{$[8, 12, 6]$} & $[7, 11, 6]$ & $[7, 11, 6]$ & $[7, 11, 6]$ & $[6, 12, 8]$ \\
        $[8, 12, 6]$ & $[8, 12, 6]$ & \central{$[8, 12, 6]$} & $[7, 12, 7]$ & $[7, 12, 7]$ \\
        $[10, 15, 7]$ & $[8, 12, 6]$ & $[8, 12, 6]$ & $[8, 12, 6]$ & \central{$[8, 12, 6]$} \\
        $[10, 15, 7]$ & $[9, 14, 7]$ & $[8, 13, 7]$ & $[8, 13, 7]$ & \\
        $[10, 15, 7]$ & $[9, 14, 7]$ & $[8, 13, 7]$ & $[8, 13, 7]$ & \\
        \central{$[12, 18, 8]$} & $[10, 15, 7]$ & $[9, 14, 7]$ & \central{$[8, 14, 8]$} & \\
        \central{$[12, 18, 8]$} & $[10, 15, 7]$ & $[9, 14, 7]$ & $[9, 15, 8]$ & \\
        $[12, 18, 8]$ & $[10, 15, 7]$ & $[10, 15, 7]$ & $[9, 15, 8]$ & \\
        $[12, 18, 8]$ & $[11, 17, 8]$ & \central{$[10, 16, 8]$} & & \\
        $[12, 18, 8]$ & $[11, 17, 8]$ & $[10, 16, 8]$ & & \\
        & $[11, 17, 8]$ & $[10, 16, 8]$ & & \\
        & $[11, 17, 8]$ & $[10, 16, 8]$ & & \\
    \end{tabular}
    \caption{$f$-vectors of all possible combinatorial types of slices of $C_4$. Slices are subdivided in distinct columns depending on how many linearly independent vertices of $C_4$ they contain. Slices with the same combinatorial type appear in more than one column. \central{Blue} $f$-vectors denote combinatorial types realizable also by central slices.}
    \label{tab:4cube_slices}
\end{table} 
The data we collected contains much more than just the number of combinatorial types of slices, but the full combinatorial data of all types of slices. For instance, one may also be interested in the $f$-vectors (or equivalently, the $h$-vectors) that can arise from slices. In particular, an open problem posed in \cite{Khovanskii06:SectionsPolytopes} asks to determine the possible $h$-vectors of generic affine slices of a polytope $P$, given the $h$-vector of $P$.
\Cref{tab:4cube_slices} displays the $f$-vectors of all types of slices of the $4$-dimensional cube, grouped according to the number of linearly independent vertices of $C_4$ that the slice contains.

\section{Color classes of graphs of slices}\label{sec:graphs}

In \cite{Nakamura1980,Fukuda1997} it was computed that $C_d$ has $12, 61$ and $484$ combinatorial types of slices for $d = 3,4,5$, respectively. The attentive reader will notice that these numbers differ from those in \Cref{thm:main_affine}. Indeed, the notion of combinatorial types used in \cite{Nakamura1980,Fukuda1997} differs from the standard definition \cite{Ziegler95:LecturesPolytopes,Gruenbaum03:Polytopes}. In \Cref{sec:graph_alg}, we explain describe this alternative notion, which we refer to as \emph{color class}, and provide a method to compute it. With this terminology we reproduce the computational results of \cite{Nakamura1980,Fukuda1997} on the numbers of color classes of graphs of slices of $C_d$ in dimensions $d = 3,4,5$. As in the previous section, these computations lead to further observations, questions, and conjectures about graphs of slices of the cube, which we present in \Cref{subsec:graphs}.

\subsection{Computation and Algorithm}\label{sec:graph_alg}
Let $P$ be a polytope and $H$ a hyperplane. Recall that every vertex of a slice $Q = P \cap H$ is either a vertex of $P$ or the intersection of $H$ with the relative interior of an edge of $P$.
We consider the vertex-edge graph $G(Q) = (V(Q), E(Q))$, where $V(Q) = \vertices(Q)$ and $E(Q)$ are the vertices and edges of the polytope $Q$, equipped with a vertex coloring 
\begin{align*}
    c : V(Q) &\to \{\text{black}, \text{white}\} \\
    v &\mapsto \begin{cases}
        \text{white} & \text{if } v \in \vertices(P) \ , \\
        \text{black} & \text{otherwise} \ .
    \end{cases}
\end{align*}
A \emph{color-preserving graph isomorphism} between two such colored graphs $G(Q_1) = (V(Q_1),$ $E(Q_1), c_1)$ and $G(Q_2) = (V(Q_2),E(Q_2), c_2)$ and is a bijection $\varphi: V(Q_1) \to V(Q_2)$ such that 
\begin{enumerate}[(i)]
    \item $\varphi$ is a graph isomorphism, i.e., $(v,w) \in E(Q_1)$ if and only if $(\varphi(v), \varphi(w)) \in E(Q_2)$, and
    \item $\varphi$ is preserves colors, i.e., $c_1(v) = c_2(\varphi(v))$ for all $v \in V(Q_1)$.
\end{enumerate}
In this language, the results of \cite{Nakamura1980, Fukuda1997} count the number of colored graphs of slices of the cube, up to color-preserving graph isomorphism. We call each such isomorphism class a \emph{color class}. 
\Cref{fig:color-classes} illustrates all color classes of slices of $C_3$.
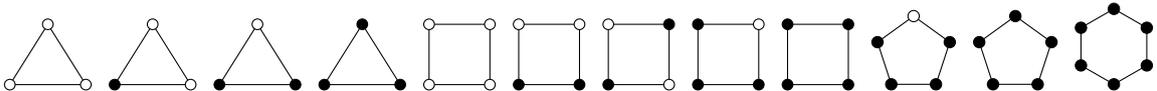
\begin{figure}[b]
    \centering
    \begin{tikzpicture}[scale=1]
    \draw (0,0) -- (1,0) -- (0.5,0.8) -- (0,0);
    \draw[fill=white] (0,0) circle (2pt);
    \draw[fill=white] (1,0) circle (2pt);
    \draw[fill=white] (0.5,0.8) circle (2pt);
\end{tikzpicture}\hspace{0.25em}
\begin{tikzpicture}[scale=1]
    \draw (0,0) -- (1,0) -- (0.5,0.8) -- (0,0);
    \draw[fill=black] (0,0) circle (2pt);
    \draw[fill=white] (1,0) circle (2pt);
    \draw[fill=white] (0.5,0.8) circle (2pt);
\end{tikzpicture}\hspace{0.25em}
\begin{tikzpicture}[scale=1]
    \draw (0,0) -- (1,0) -- (0.5,0.8) -- (0,0);
    \draw[fill=black] (0,0) circle (2pt);
    \draw[fill=black] (1,0) circle (2pt);
    \draw[fill=white] (0.5,0.8) circle (2pt);
\end{tikzpicture}\hspace{0.25em}
\begin{tikzpicture}[scale=1]
    \draw (0,0) -- (1,0) -- (0.5,0.8) -- (0,0);
    \draw[fill=black] (0,0) circle (2pt);
    \draw[fill=black] (1,0) circle (2pt);
    \draw[fill=black] (0.5,0.8) circle (2pt);
\end{tikzpicture}\hspace{0.25em}
\begin{tikzpicture}[scale=0.8]
    \draw (0,0) -- (1,0) -- (1,1) -- (0,1) -- (0,0);
    \draw[fill=white] (0,0) circle (2.5pt);
    \draw[fill=white] (1,0) circle (2.5pt);
    \draw[fill=white] (1,1) circle (2.5pt);
    \draw[fill=white] (0,1) circle (2.5pt);
\end{tikzpicture}\hspace{0.25em}
\begin{tikzpicture}[scale=0.8]
    \draw (0,0) -- (1,0) -- (1,1) -- (0,1) -- (0,0);
    \draw[fill=black] (0,0) circle (2.5pt);
    \draw[fill=black] (1,0) circle (2.5pt);
    \draw[fill=white] (1,1) circle (2.5pt);
    \draw[fill=white] (0,1) circle (2.5pt);
\end{tikzpicture}\hspace{0.25em}
\begin{tikzpicture}[scale=0.8]
    \draw (0,0) -- (1,0) -- (1,1) -- (0,1) -- (0,0);
    \draw[fill=black] (0,0) circle (2.5pt);
    \draw[fill=white] (1,0) circle (2.5pt);
    \draw[fill=black] (1,1) circle (2.5pt);
    \draw[fill=white] (0,1) circle (2.5pt);
\end{tikzpicture}\hspace{0.25em}
\begin{tikzpicture}[scale=0.8]
    \draw (0,0) -- (1,0) -- (1,1) -- (0,1) -- (0,0);
    \draw[fill=black] (0,0) circle (2.5pt);
    \draw[fill=black] (1,0) circle (2.5pt);
    \draw[fill=white] (1,1) circle (2.5pt);
    \draw[fill=black] (0,1) circle (2.5pt);
\end{tikzpicture}\hspace{0.25em}
\begin{tikzpicture}[scale=0.8]
    \draw (0,0) -- (1,0) -- (1,1) -- (0,1) -- (0,0);
    \draw[fill=black] (0,0) circle (2.5pt);
    \draw[fill=black] (1,0) circle (2.5pt);
    \draw[fill=black] (1,1) circle (2.5pt);
    \draw[fill=black] (0,1) circle (2.5pt);
\end{tikzpicture}\hspace{0.25em}
\begin{tikzpicture}%
	[scale=0.5,
	back/.style={loosely dotted, thin},
	edge/.style={color=black},
	black/.style={inner sep=1.5pt,circle,draw=black,fill=black},
    white/.style={inner sep=1.5pt,circle,draw=black,fill=white}]

\coordinate (0.58779, -0.80902) at (0.58779, -0.80902);
\coordinate (0.95106, 0.30902) at (0.95106, 0.30902);
\coordinate (0.00000, 1.00000) at (0.00000, 1.00000);
\coordinate (-0.58779, -0.80902) at (-0.58779, -0.80902);
\coordinate (-0.95106, 0.30902) at (-0.95106, 0.30902);

%%
%%
%% Drawing edges
%%
\draw[edge] (0.58779, -0.80902) -- (0.95106, 0.30902);
\draw[edge] (0.58779, -0.80902) -- (-0.58779, -0.80902);
\draw[edge] (0.95106, 0.30902) -- (0.00000, 1.00000);
\draw[edge] (0.00000, 1.00000) -- (-0.95106, 0.30902);
\draw[edge] (-0.58779, -0.80902) -- (-0.95106, 0.30902);
%%
%%
%% Drawing the vertices
%%
\node[black] at (0.58779, -0.80902)     {};
\node[black] at (0.95106, 0.30902)     {};
\node[white] at (0.00000, 1.00000)     {};
\node[black] at (-0.58779, -0.80902)     {};
\node[black] at (-0.95106, 0.30902)     {};
\end{tikzpicture}\hspace{0.25em}
\begin{tikzpicture}%
	[scale=0.5,
	back/.style={loosely dotted, thin},
	edge/.style={color=black},
	black/.style={inner sep=1.5pt,circle,draw=black,fill=black},
    white/.style={inner sep=1.5pt,circle,draw=black,fill=white}]

\coordinate (0.58779, -0.80902) at (0.58779, -0.80902);
\coordinate (0.95106, 0.30902) at (0.95106, 0.30902);
\coordinate (0.00000, 1.00000) at (0.00000, 1.00000);
\coordinate (-0.58779, -0.80902) at (-0.58779, -0.80902);
\coordinate (-0.95106, 0.30902) at (-0.95106, 0.30902);

%%
%%
%% Drawing edges
%%
\draw[edge] (0.58779, -0.80902) -- (0.95106, 0.30902);
\draw[edge] (0.58779, -0.80902) -- (-0.58779, -0.80902);
\draw[edge] (0.95106, 0.30902) -- (0.00000, 1.00000);
\draw[edge] (0.00000, 1.00000) -- (-0.95106, 0.30902);
\draw[edge] (-0.58779, -0.80902) -- (-0.95106, 0.30902);
%%
%%
%% Drawing the vertices
%%
\node[black] at (0.58779, -0.80902)     {};
\node[black] at (0.95106, 0.30902)     {};
\node[black] at (0.00000, 1.00000)     {};
\node[black] at (-0.58779, -0.80902)     {};
\node[black] at (-0.95106, 0.30902)     {};
\end{tikzpicture}\hspace{0.25em}
\begin{tikzpicture}%
	[scale=0.5,
	back/.style={loosely dotted, thin},
	edge/.style={color=black},
	vertex/.style={inner sep=1.5pt,circle,draw=black,fill=black},
    white/.style={inner sep=1.5pt,circle,draw=black,fill=white}]

%% Coordinate of the vertices:
%%
\coordinate (0.86603, -0.50000) at (0.86603, -0.50000);
\coordinate (0.86603, 0.50000) at (0.86603, 0.50000);
\coordinate (0.00000, -1.00000) at (0.00000, -1.00000);
\coordinate (0.00000, 1.00000) at (0.00000, 1.00000);
\coordinate (-0.86603, -0.50000) at (-0.86603, -0.50000);
\coordinate (-0.86603, 0.50000) at (-0.86603, 0.50000);

%%
%%
%% Drawing edges
%%
\draw[edge] (0.86603, -0.50000) -- (0.86603, 0.50000);
\draw[edge] (0.86603, -0.50000) -- (0.00000, -1.00000);
\draw[edge] (0.86603, 0.50000) -- (0.00000, 1.00000);
\draw[edge] (0.00000, -1.00000) -- (-0.86603, -0.50000);
\draw[edge] (0.00000, 1.00000) -- (-0.86603, 0.50000);
\draw[edge] (-0.86603, -0.50000) -- (-0.86603, 0.50000);
%%
%%
%% Drawing the vertices
%%
\node[vertex] at (0.86603, -0.50000)     {};
\node[vertex] at (0.86603, 0.50000)     {};
\node[vertex] at (0.00000, -1.00000)     {};
\node[vertex] at (0.00000, 1.00000)     {};
\node[vertex] at (-0.86603, -0.50000)     {};
\node[vertex] at (-0.86603, 0.50000)     {};
\end{tikzpicture}
    \caption{The $12$ color classes of graphs of slices of $C_3$.}
    \label{fig:color-classes}
\end{figure}

We have verified computationally that the numbers reported in \cite{Nakamura1980,Fukuda1997} are indeed the correct counts of such isomorphism classes for $d=3,4,5$.
In our computations, we largely follow their general approach, based on the observation that any affine hyperplane $H = \{x \in \R^d \mid \langle u,x \rangle + a = 0 \}$ induces a vertex labeling $\ell$ of the vertex-edge graph of $P$ with labels $\{+,0,-\}$ via $\ell(v) = \sgn(\langle u, v \rangle + a)$. 
Our goal is to determine all vertex labelings $\ell : V(P) \to \{-,0,+\}$ that are \emph{geometrically realizable}, i.e., for which there exists a hyperplane $H = \{x\in\R^d \mid \langle x, u \rangle + a = 0\}$ satisfying $\ell(v) = \sgn(\langle u, v \rangle + a)$ for every $v\in V(P)$. To this end we collect four necessary conditions for a labeling $\ell$ to be geometrically realizable.

First, convexity of $P$ implies that for every labeling the induced subgraph $S$ generated by the vertices with label $+$ is connected, and the same holds for the induced subgraph of vertices labeled $-$. As the slice is unchanged when interchanging the labels $+$ and $-$, we may assume $|V(S)|\leq \frac{1}{2}|V(P)|$.
Second, every vertex with label $0$ is adjacent to at least one vertex with label $+$, unless the hyperplane supports a face of $P$ (in which case no vertex is labeled $+$). In other words, every vertex with label $0$ lies in the neighborhood 
\[
N(S) = \{v\in V(P) \setminus V(S) \mid \exists \, w \in V(S) : (v,w) \in E(P)\}.
\]
Third, the list of \emph{forbidden squares} in \Cref{fig:forbidden-squares} specifies all labelings of the $4$-cycles in the cube graph $G(C_d)$ that are not geometrically realizable. It is straightforward to verify that this list is complete: every other square labeling occurs in some geometrically realizable labeling.

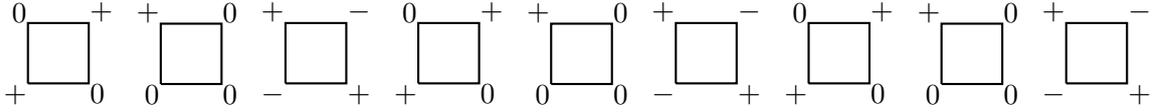
\begin{figure}
    \centering
    \begin{tikzpicture}[scale=0.8, inner sep=0.1mm]
    \draw[thick] (0,0) -- (1,0) -- (1,1) -- (0,1) -- (0,0);
    \node[anchor=north east] at (0,0) {$+$};
    \node[anchor=north west] at (1,0) {$0$};
    \node[anchor=south west] at (1,1) {$+$};
    \node[anchor=south east] at (0,1) {$0$};
\end{tikzpicture}
\hspace{0.1em}
\begin{tikzpicture}[scale=0.8, inner sep=0.1mm]
    \draw[thick] (0,0) -- (1,0) -- (1,1) -- (0,1) -- (0,0);
    \node[anchor=north east] at (0,0) {$0$};
    \node[anchor=north west] at (1,0) {$0$};
    \node[anchor=south west] at (1,1) {$0$};
    \node[anchor=south east] at (0,1) {$+$};
\end{tikzpicture}
\hspace{0.1em}
\begin{tikzpicture}[scale=0.8, inner sep=0.1mm]
    \draw[thick] (0,0) -- (1,0) -- (1,1) -- (0,1) -- (0,0);
    \node[anchor=north east] at (0,0) {$-$};
    \node[anchor=north west] at (1,0) {$+$};
    \node[anchor=south west] at (1,1) {$-$};
    \node[anchor=south east] at (0,1) {$+$};
\end{tikzpicture}
\hspace{0.1em}
\begin{tikzpicture}[scale=0.8, inner sep=0.1mm]
    \draw[thick] (0,0) -- (1,0) -- (1,1) -- (0,1) -- (0,0);
    \node[anchor=north east] at (0,0) {$+$};
    \node[anchor=north west] at (1,0) {$0$};
    \node[anchor=south west] at (1,1) {$+$};
    \node[anchor=south east] at (0,1) {$0$};
\end{tikzpicture}
\hspace{0.1em}
\begin{tikzpicture}[scale=0.8, inner sep=0.1mm]
    \draw[thick] (0,0) -- (1,0) -- (1,1) -- (0,1) -- (0,0);
    \node[anchor=north east] at (0,0) {$0$};
    \node[anchor=north west] at (1,0) {$0$};
    \node[anchor=south west] at (1,1) {$0$};
    \node[anchor=south east] at (0,1) {$+$};
\end{tikzpicture}
\hspace{0.1em}
\begin{tikzpicture}[scale=0.8, inner sep=0.1mm]
    \draw[thick] (0,0) -- (1,0) -- (1,1) -- (0,1) -- (0,0);
    \node[anchor=north east] at (0,0) {$-$};
    \node[anchor=north west] at (1,0) {$+$};
    \node[anchor=south west] at (1,1) {$-$};
    \node[anchor=south east] at (0,1) {$+$};
\end{tikzpicture}
\hspace{0.1em}
\begin{tikzpicture}[scale=0.8, inner sep=0.1mm]
    \draw[thick] (0,0) -- (1,0) -- (1,1) -- (0,1) -- (0,0);
    \node[anchor=north east] at (0,0) {$+$};
    \node[anchor=north west] at (1,0) {$0$};
    \node[anchor=south west] at (1,1) {$+$};
    \node[anchor=south east] at (0,1) {$0$};
\end{tikzpicture}
\hspace{0.1em}
\begin{tikzpicture}[scale=0.8, inner sep=0.1mm]
    \draw[thick] (0,0) -- (1,0) -- (1,1) -- (0,1) -- (0,0);
    \node[anchor=north east] at (0,0) {$0$};
    \node[anchor=north west] at (1,0) {$0$};
    \node[anchor=south west] at (1,1) {$0$};
    \node[anchor=south east] at (0,1) {$+$};
\end{tikzpicture}
\hspace{0.1em}
\begin{tikzpicture}[scale=0.8, inner sep=0.1mm]
    \draw[thick] (0,0) -- (1,0) -- (1,1) -- (0,1) -- (0,0);
    \node[anchor=north east] at (0,0) {$-$};
    \node[anchor=north west] at (1,0) {$+$};
    \node[anchor=south west] at (1,1) {$-$};
    \node[anchor=south east] at (0,1) {$+$};
\end{tikzpicture}
    \caption{The 9 forbidden squares from \cite{Fukuda1997}.}
    \label{fig:forbidden-squares}
\end{figure}

These necessary conditions are essentially those described in \cite{Fukuda1997}.
However, not every labeling that satisfies them is geometrically realizable. For a given labeling $\ell$, the set of hyperplanes realizing it forms a relatively open polyhedral cone, defined by one equation for each vertex labeled $0$ and by one strict inequality for each remaining vertex. Geometric realizability can therefore be tested by computing the closed cone $C(\ell)$ defined by non-strict inequalities (and equations), and comparing its dimension with the expected one, yielding the fourth condition.
With this observation we obtain the following sketch of an algorithm for checking realizability. We point out that we do not know whether, or in what way, a check for this fourth condition was performed in the computations reported in \cite{Fukuda1997}, but it proved necessary in our computation in order to obtain correct labellings (see \Cref{ex:invalid-labelling}).

In the actual implementation of \Cref{alg:graph}, carried out in \texttt{SageMath 10.5} \cite{sagemath}, we include checks to avoid repeating computations on labelled graphs that are identical up to labelled graph isomorphism, taking advantage of the high symmetry of the cube. The $12$ color classes of graphs of slices of $C_3$ are depicted in \Cref{fig:color-classes}. The $61$ and $484$ classes for $C_4$ and $C_5$, respectively, are printed in \cite{Fukuda1997}. For detailed inspection of these $12+61+484$ classes, we provide the complete data set of color classes as raw data in our repository.

\begin{example}[Necessary conditions are not sufficient]
\label{ex:invalid-labelling}
Consider the cube $C_3$ labeled as 
\begin{center}
    \begin{tikzpicture}[inner sep = 0.3pt, scale = 0.8]
    \draw[thick] (0+0.5,0+0.3) -- (1+0.5,0+0.3) -- (1+0.5,1+0.3) -- (0+0.5,1+0.3) -- (0+0.5,0+0.3);
    \draw[line width=0.5em, white] (0,0) -- (1,0) -- (1,1) -- (0,1) -- (0,0);
    \draw[thick] (0,0) -- (1,0) -- (1,1) -- (0,1) -- (0,0);
    \draw (0,0) -- (0.5,0.3);
    \draw (1,0) -- (1.5,0.3);
    \draw (0,1) -- (0.5,1.3);
    \draw (1,1) -- (1.5,1.3);
    \node[anchor = north east] at (0,0) {$-$};
    \node[anchor = north west] at (1,0) {$-$};
    \node[anchor = south east] at (0,1) {$-$};
    \node[anchor = north west] at (1.05,1) {$0$};
    \node[anchor = south east] at (0.45,0.3) {$0$};
    \node[anchor = north west] at (1.5,0.3) {$-$};
    \node[anchor = south east] at (0.5,1.3) {$0$};
    \node[anchor = north west] at (1.5,1.3) {$\ -$};
\end{tikzpicture}
\end{center}
This labeling satisfies all three necessary conditions from \cite{Fukuda1997}: the induced subgraphs of vertices with labels $+$ and $-$ are connected; every vertex labeled $0$ is adjacent to a vertex labeled $+$; no forbidden square occurs. Nevertheless, the unique hyperplane containing the vertices with label $0$ necessarily passes through four vertices of the cube, so this labeling is not geometrically realizable.
\end{example}

\begin{algorithm}
    \Require Polytope $C_d = [-1,1]^d$
    \Ensure List of colored vertex-edge graphs of slices of $C_d$ 
        \begin{algorithmic}[1]
        \State Compute vertex-edge graph $G(C_d) = (V(C_d),E(C_d))$
        \State \texttt{color\_classes} $\gets$ empty
        \For{connected induced subgraphs $S$ of $G(C_d)$ of size $|V(S)| \leq \frac{1}{2}|V(C_d)|$}
        \State compute the neighborhood $N(S)$
        \State assign label $+$ to all vertices in $S$
        \State assign label $-$ to all vertices in $V(C_d) \setminus \{V(S) \cup N(S)\}$
        \For{$\{-,0\}$-labelings of $N(S)$}
        \If{the labeling $\ell$ of $V(G)$ does not contain a forbidden square}
            \State compute the cone $C(\ell)$
            \If{$C(\ell)$ has relative interior of the expected dimension}
                \State sample a hyperplane $H$ from the relative interior of the cone
                \State compute the graph $K$ of $[-1,1]^d \cap H$
                \If{$K$ not in \texttt{color\_classes} (up to colored graph isomorphism)}
                    \State store the colored graph of the slice 
        \EndIf
        \EndIf
        \EndIf
        \EndFor
        \EndFor
        \end{algorithmic}
    \caption{Graph algorithm}
    \label{alg:graph}
\end{algorithm}

\subsection{Graphs of slices of the cube}\label{subsec:graphs}

In dimensions $4$ and higher, polytopes of distinct combinatorial types (in the sense of isomorphism classes of face-lattices) can, in general, have identical vertex-edge graphs. It is therefore not expected that arbitrary slices of cubes of dimension $5$ should be distinguishable by their graphs, or their respective color class. However, analyzing the computed color classes, we observe that indeed all combinatorial types are distinguished by their vertex–edge graphs, thus justifying the computational method explained in \Cref{sec:graph_alg}.
While this holds true for $d \leq 5$,
we do not know whether this remains true for slices of cubes in higher dimensions.

\begin{question}
    Do all distinct combinatorial types of slices of the cube $C_d$ have nonisomorphic vertex-edge graphs? Equivalently, are slices of cubes reconstructible from their graphs, given the additional information that they are indeed slices of the cube?
\end{question}

We now collect evidence both for and against an affirmative answer to this question.
First, note that that every simple polytope is reconstructible from its graph (\cite{Blind1987,Kalai1988}). More generally, polytopes with at most two nonsimple vertices are reconstructible from their graphs \cite{Doolittle2018}. This is key to the following statement.

\begin{proposition}
    Any slice of the cube $C_d$ which contains at most $2$ vertices of $C_d$ is reconstructible from its graph. 
\end{proposition}
\begin{proof}
    Recall that a vertex of a slice $C_d \cap H$ is either a vertex of $C_d$ or arises as the intersection of $H$ with the relative interior of an edge $e$.
    Since any polytope with at most $2$ nonsimple vertices is reconstructible, it suffices to show that a vertex of the slice $C_d \cap H$ of the form $v = \operatorname{relint}(e) \cap H$ is simple, i.e., has degree $d-1$. 

    The edges of the slice incident to $v$ are intersections of squares of $C_d$ containing $e$. Since $H$ intersects $e$ in its relative interior, all adjacent squares are intersected in their respective relative interiors as well. We thus need to show that $e$ is contained in precisely $d-1$ squares. For $d \geq 2$ we have $C_d = C_{d-1} \times C_1$, and by symmetry it suffices to check the statement for the edge $e = (-1,\dots,-1,\lambda)$, $\lambda\in[-1,1]$. The squares of $C_d$ containing $e$ are precisely the sets $\tilde{e} \times C_1$, where $\tilde{e}$ is an edge of $C_{d-1} \times \{-1\}$ containing the vertex $(-1,\dots,-1)\in C_{d-1}$. As $C_{d-1}$ is simple, the number of such edges is $d-1$.
\end{proof}

On the other hand, \cite[Example 23]{PinedaVillavicencio2022} provide a polytope $Q$ whose graph is that of the hypersimplex $\Delta(2,5)$ -- itself a slice of $C_5$ -- but $Q$ and $\Delta(2,5)$ are not combinatorially equivalent. 
The key point is that $Q$ cannot be realized as a slice of a cube. Therefore, the additional information that the graph comes from a slice of a cube is genuinely needed to reconstruct such slices from their graphs up to dimension $5$. It remains open whether this is true in all dimensions. Likewise, the more restrictive question of whether a slice is reconstructible from its \emph{colored} graph is also unresolved in general.

Finally, although the $d$-dimensional cube is simple -- therefore uniquely determined by its vertex-edge graph as a $d$-dimensional polytope -- there exists, for every $k$ with $3\le k<d$, a $k$-dimensional polytope $P$ whose graph is isomorphic to that of the $d$-cube \cite{Joswig2000}. More generally, for integers $r,k$ satisfying $2r + 2 \leq k \leq d$, there exists a $k$-polytope whose $r$-skeleton is combinatorially equivalent to that of the $d$-cube.
\begin{question}
    Let $P$ be a $k$-dimensional polytope whose graph (or more generally whose $r$-skeleton) is isomorphic to that of the $d$-dimensional cube. How do the graphs of $(k-1)$-dimensional slices of $P$ compare with the graphs of the $(d-1)$-dimensional slices of the cube?
\end{question}

\printbibliography

~\\
\vfill
\noindent \textsc{Marie-Charlotte Brandenburg} \\
\textsc{Ruhr-Universität Bochum} \\
 \url{marie-charlotte.brandenburg@rub.de} \\

\noindent \textsc{Chiara Meroni} \\
\textsc{ETH Institute for Theoretical Studies} \\ 
\url{chiara.meroni@eth-its.ethz.ch} \\

\newpage
\begin{appendices}
\crefalias{section}{appendix}
\section{Distribution of slices}\label{app:distribution}

We show the distribution of the combinatorial types of slices according to the number of vertices of the slice. For the cubes $C_5$ (see \Cref{fig:distribution_5cube_all}) and $C_6$ (see \Cref{fig:distribution_6cube_all}), we display several classes of slices: all affine slices (light green), affine generic slices (dark green), all central slices (light blue), and generic central slices (dark blue). This analysis for the cube $C_4$ was already shown in \Cref{fig:4_cube_vertices}, with the same color-coding. For the cube $C_7$, we display all central slices (light blue) and generic central slices (dark blue) in \Cref{fig:distribution_7cube_all}.

\begin{figure}[!ht]
    \centering
    \rotatebox{90}{\hspace{3.5em} affine}
    \includegraphics[width=0.47\linewidth]{Figures/5cubeVertexDistribution_affine.png}
    \includegraphics[width=0.47\linewidth]{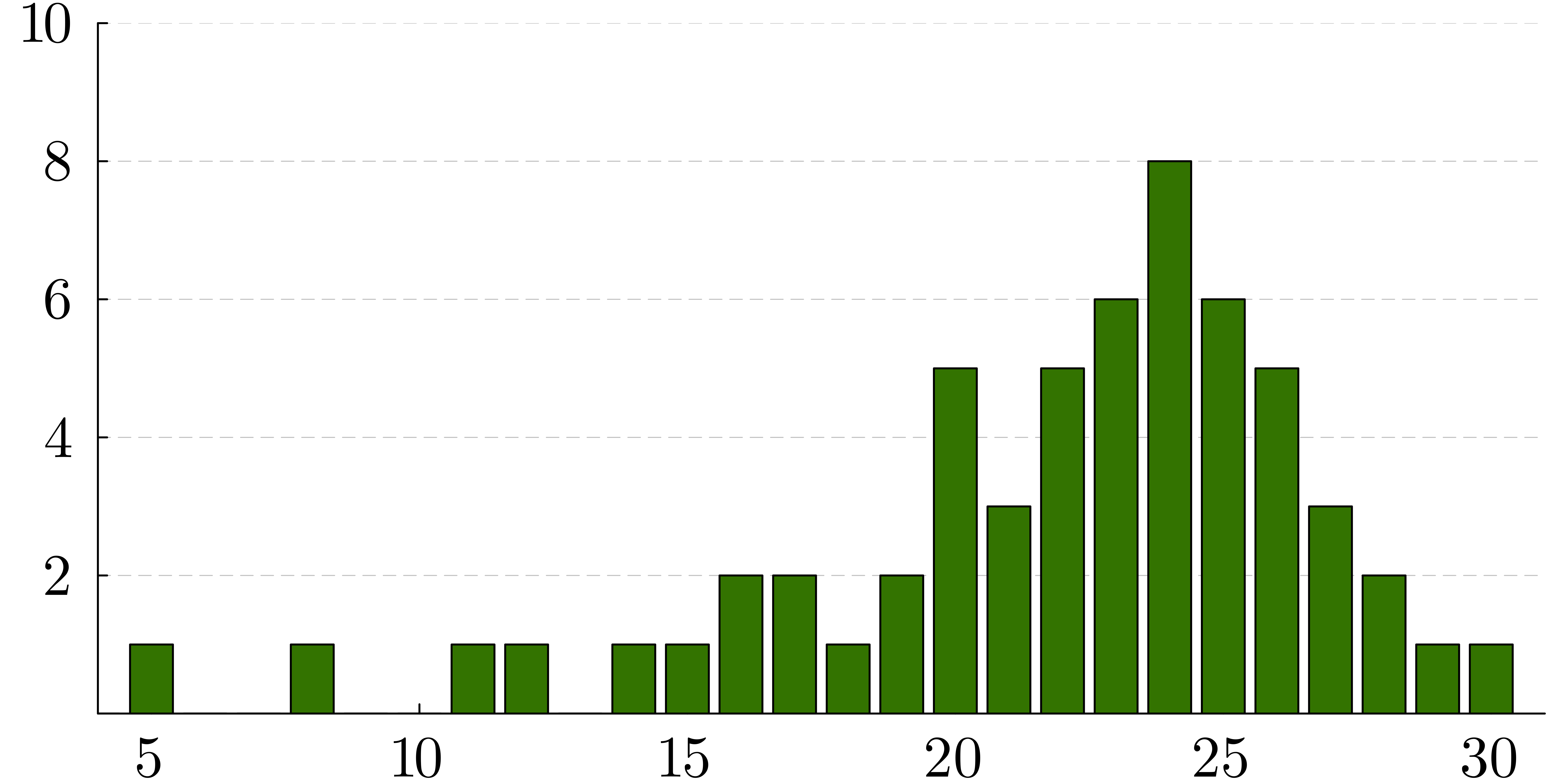}
    \\
    \rotatebox{90}{\hspace{3.4em} central}
    \includegraphics[width=0.47\linewidth]{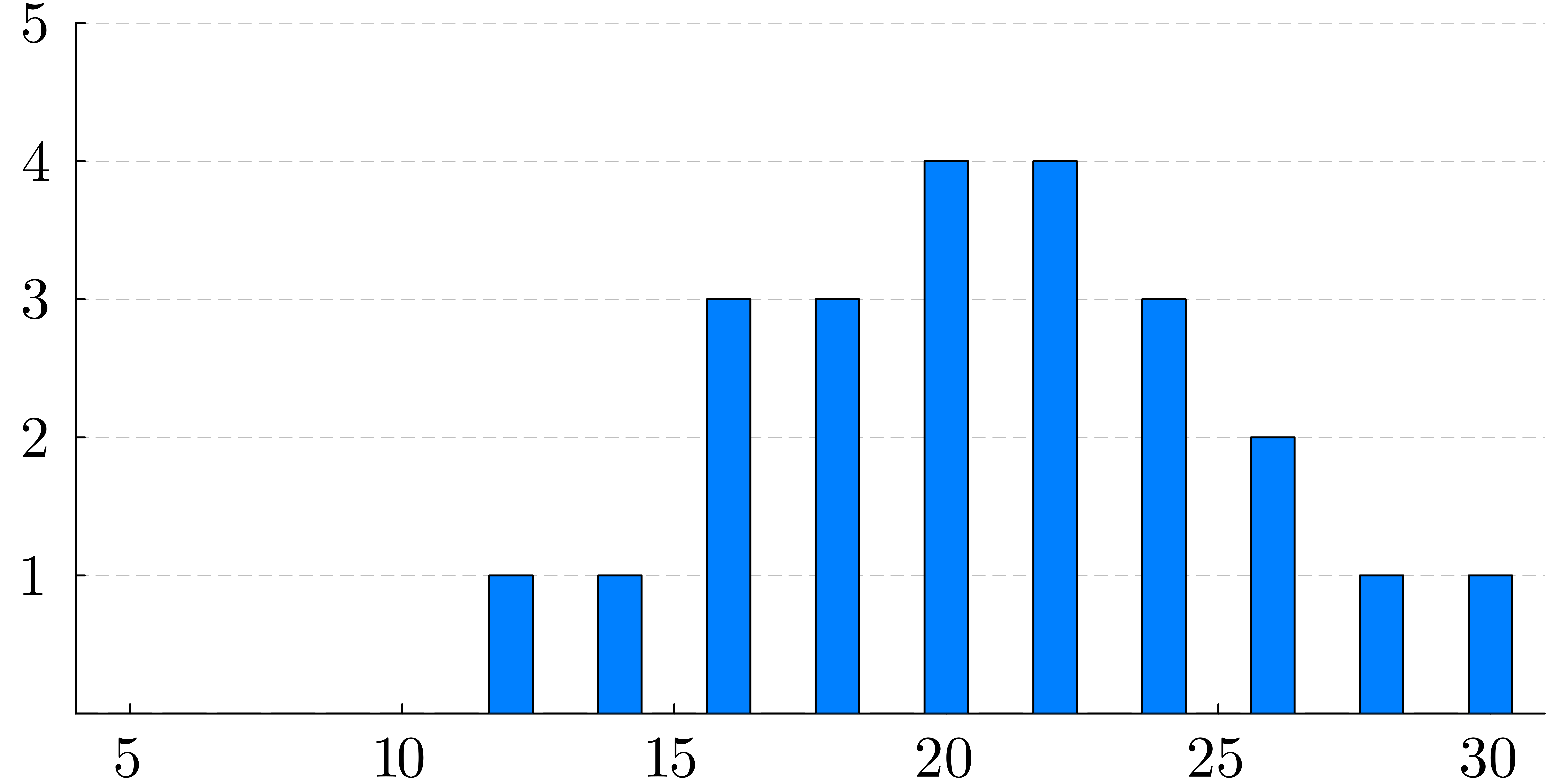}
    \includegraphics[width=0.47\linewidth]{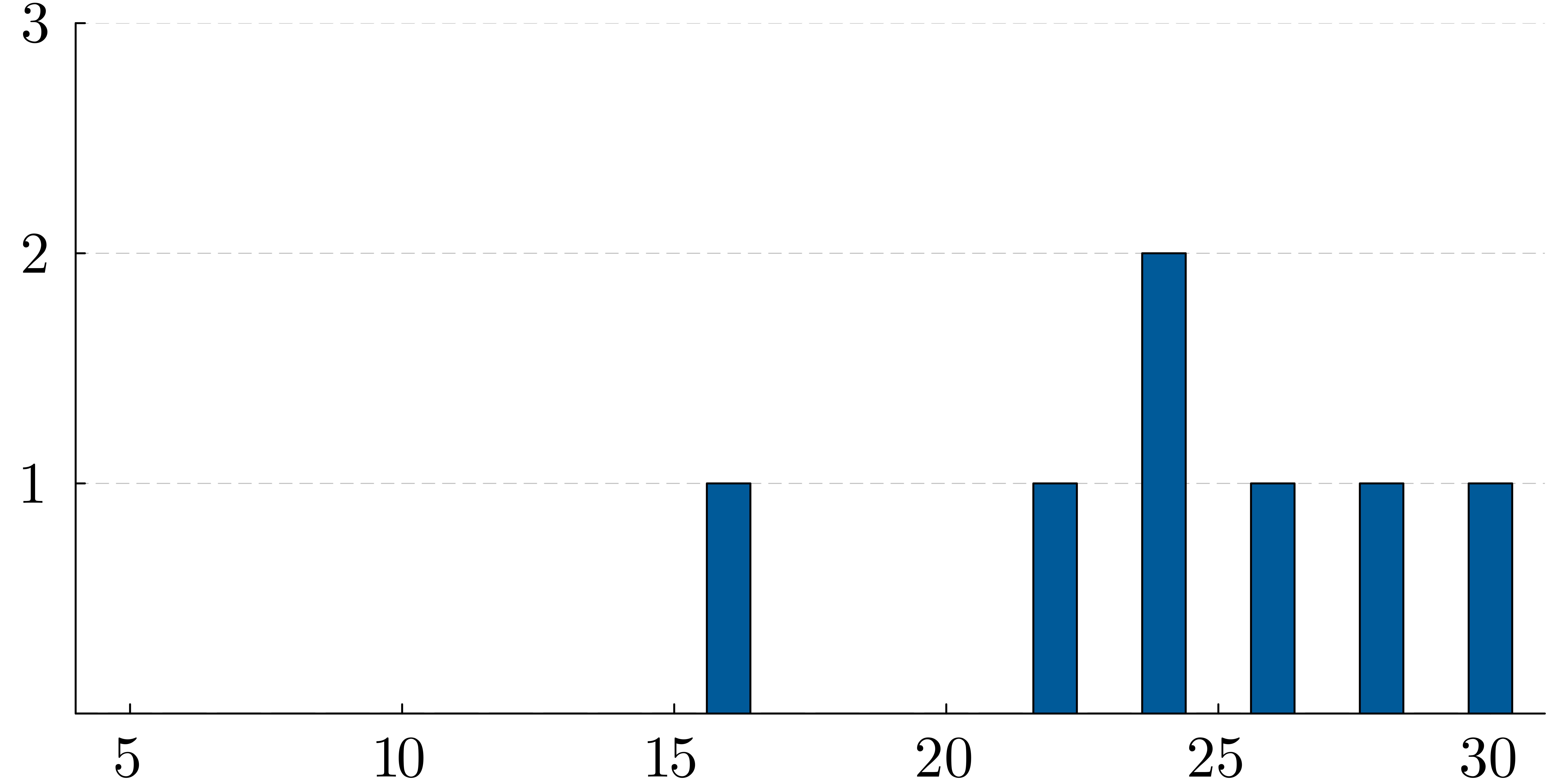}
    \caption{Number of combinatorial types of slices of $C_5$ by number of vertices. Notice that the vertical axis has a different scaling in each plot.}
    \label{fig:distribution_5cube_all}
\end{figure}

\begin{figure}[!ht]
    \centering
    \rotatebox{90}{\hspace{3.5em} affine}
    \includegraphics[width=0.47\linewidth]{Figures/6cubeVertexDistribution_affine.png}
    \includegraphics[width=0.47\linewidth]{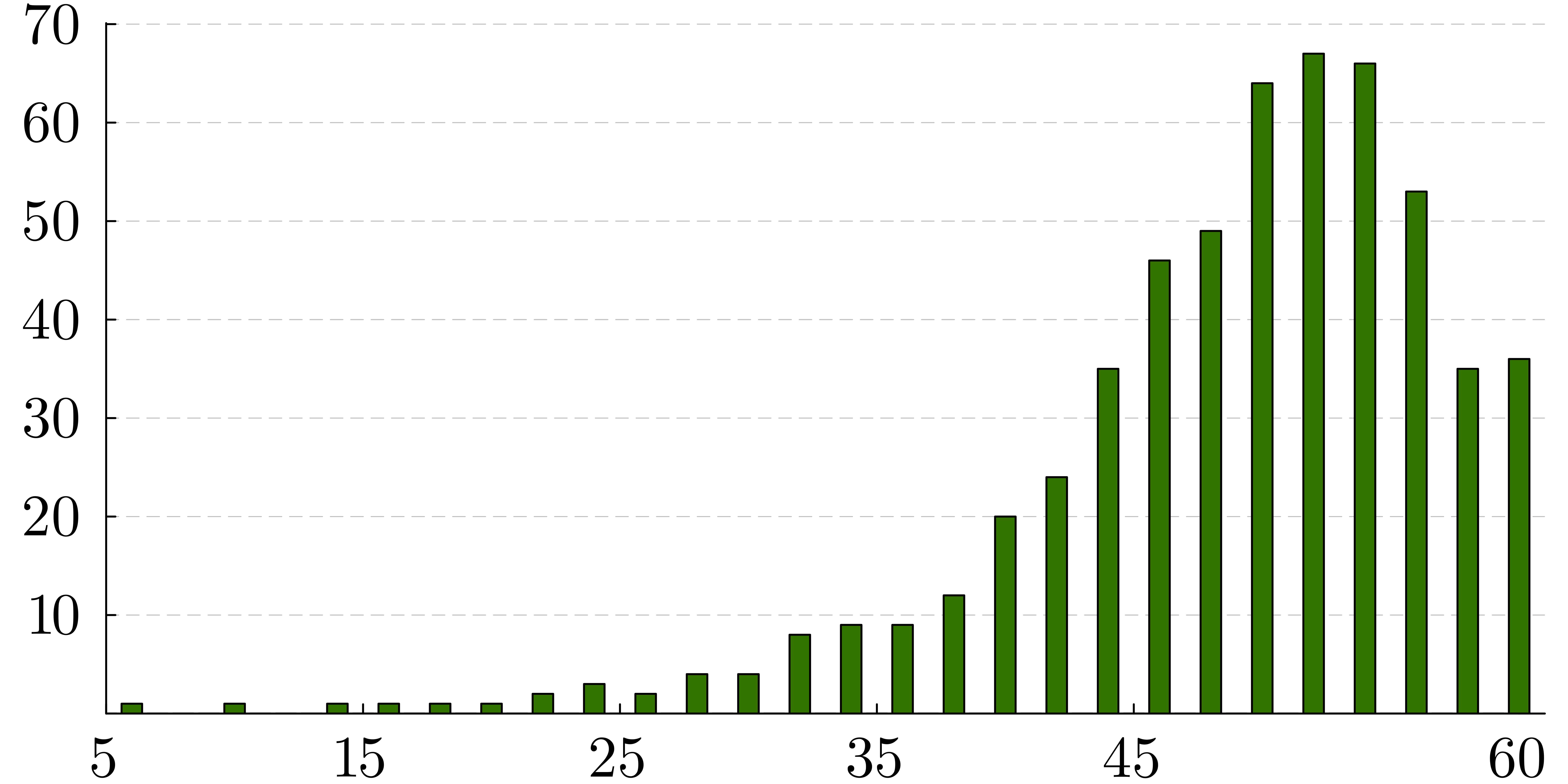}
    \\
    \rotatebox{90}{\hspace{3.4em} central}
    \includegraphics[width=0.47\linewidth]{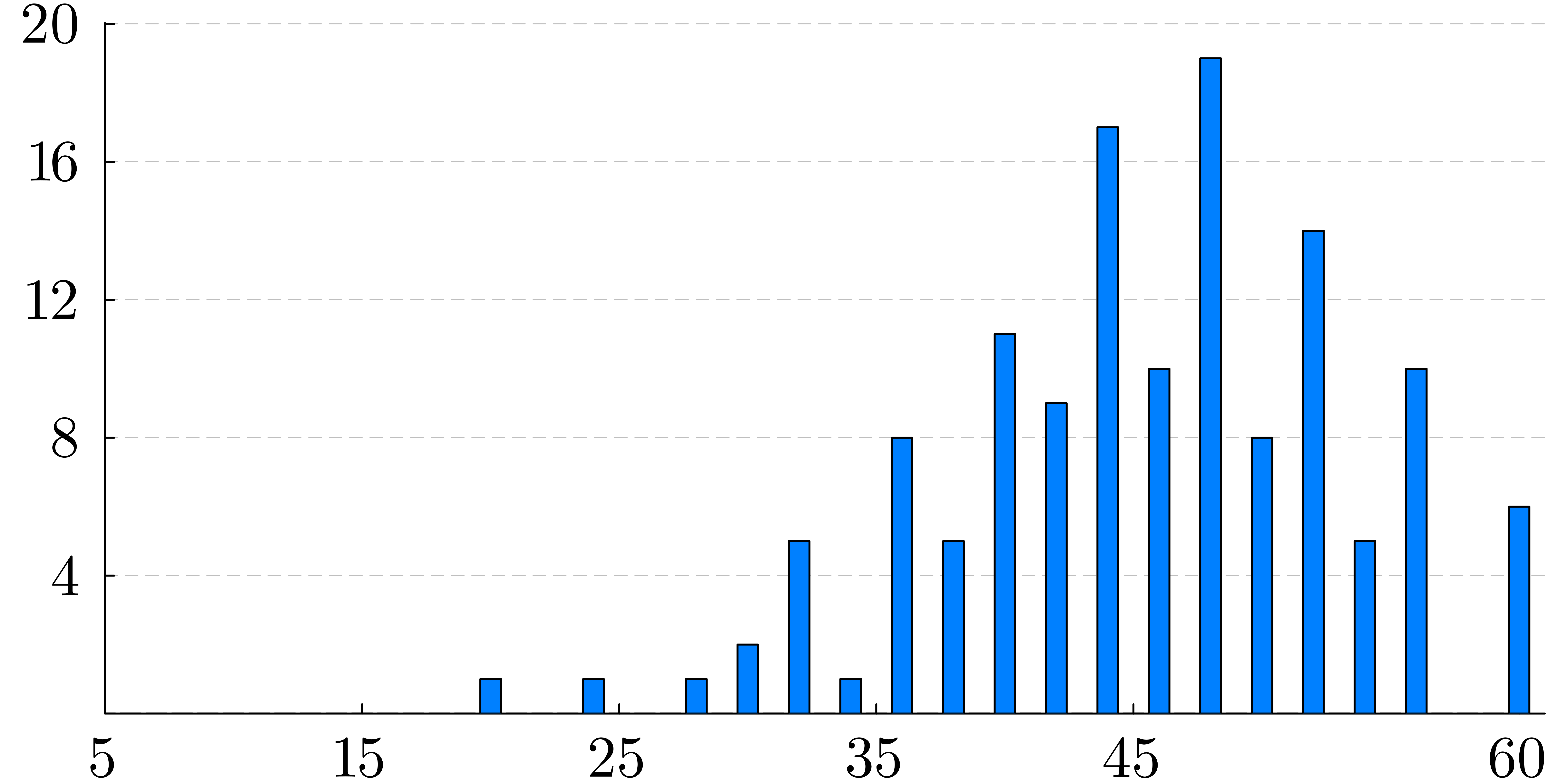}
    \includegraphics[width=0.47\linewidth]{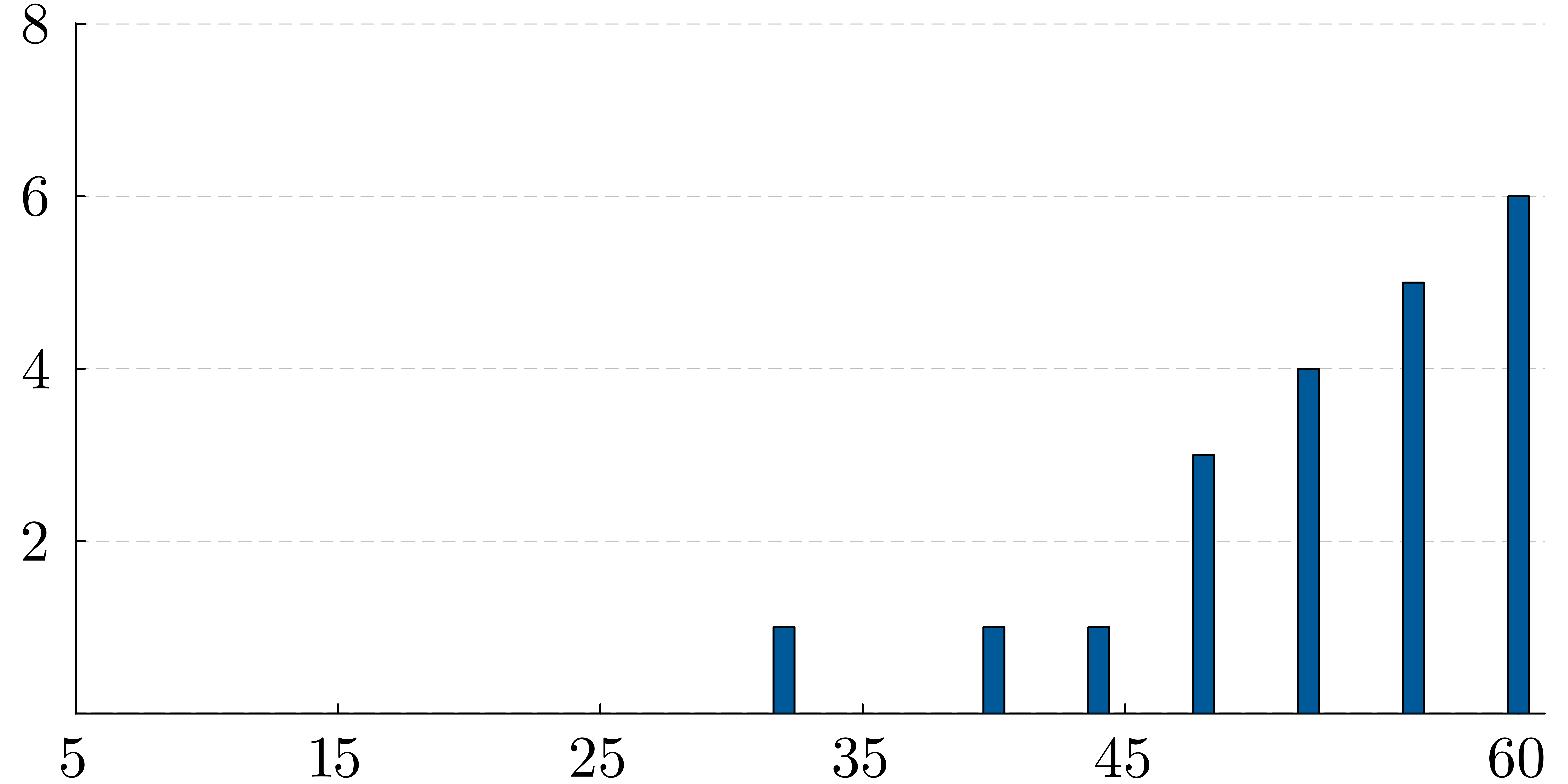}
    \caption{Number of combinatorial types of slices of $C_6$ by number of vertices. Notice that the vertical axis has a different scaling in each plot.}
    \label{fig:distribution_6cube_all}
\end{figure}

\begin{figure}[!ht]
    \centering
    \rotatebox{90}{\hspace{3.4em} central}
    \includegraphics[width=0.47\linewidth]{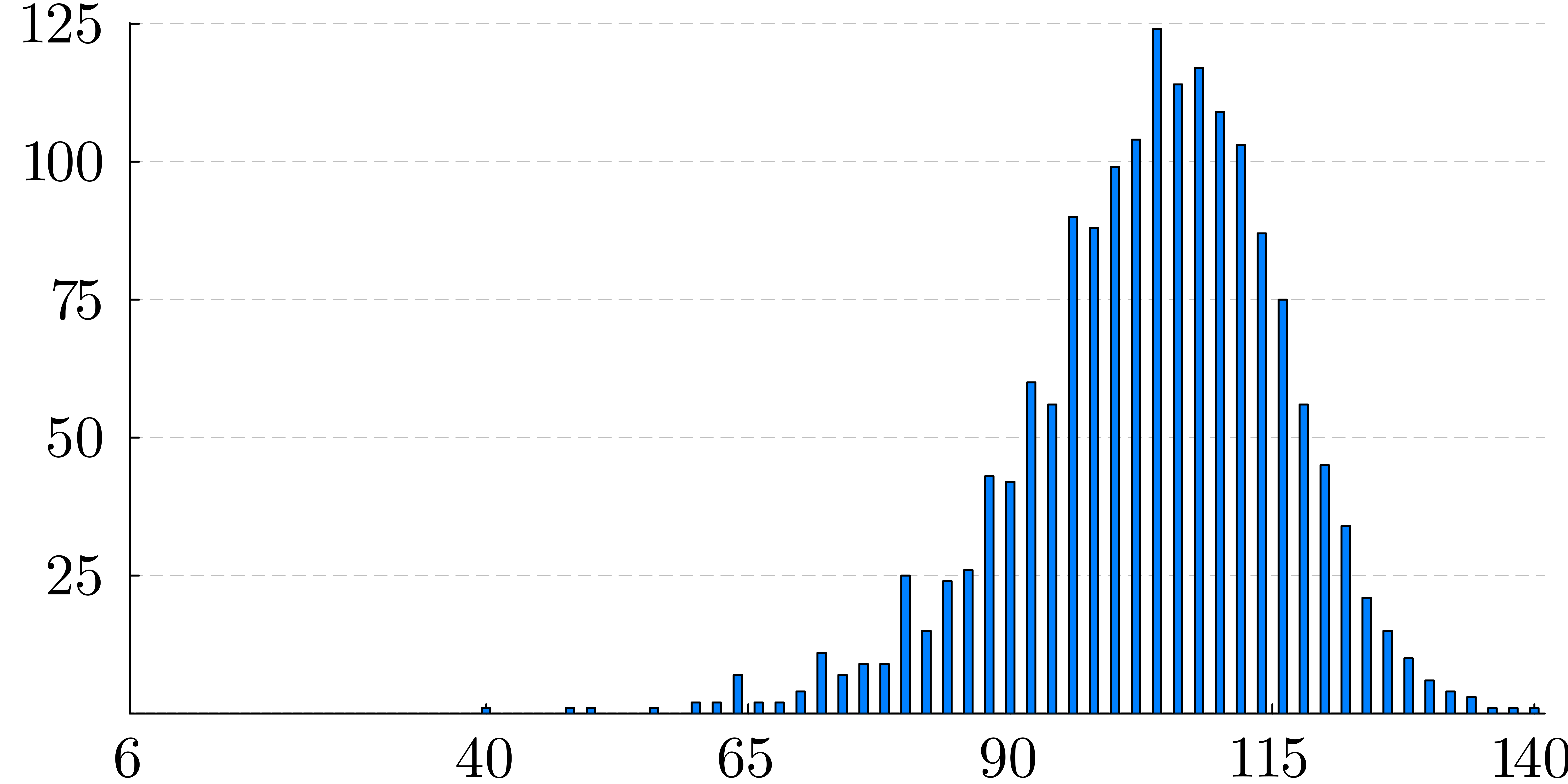}
    \includegraphics[width=0.47\linewidth]{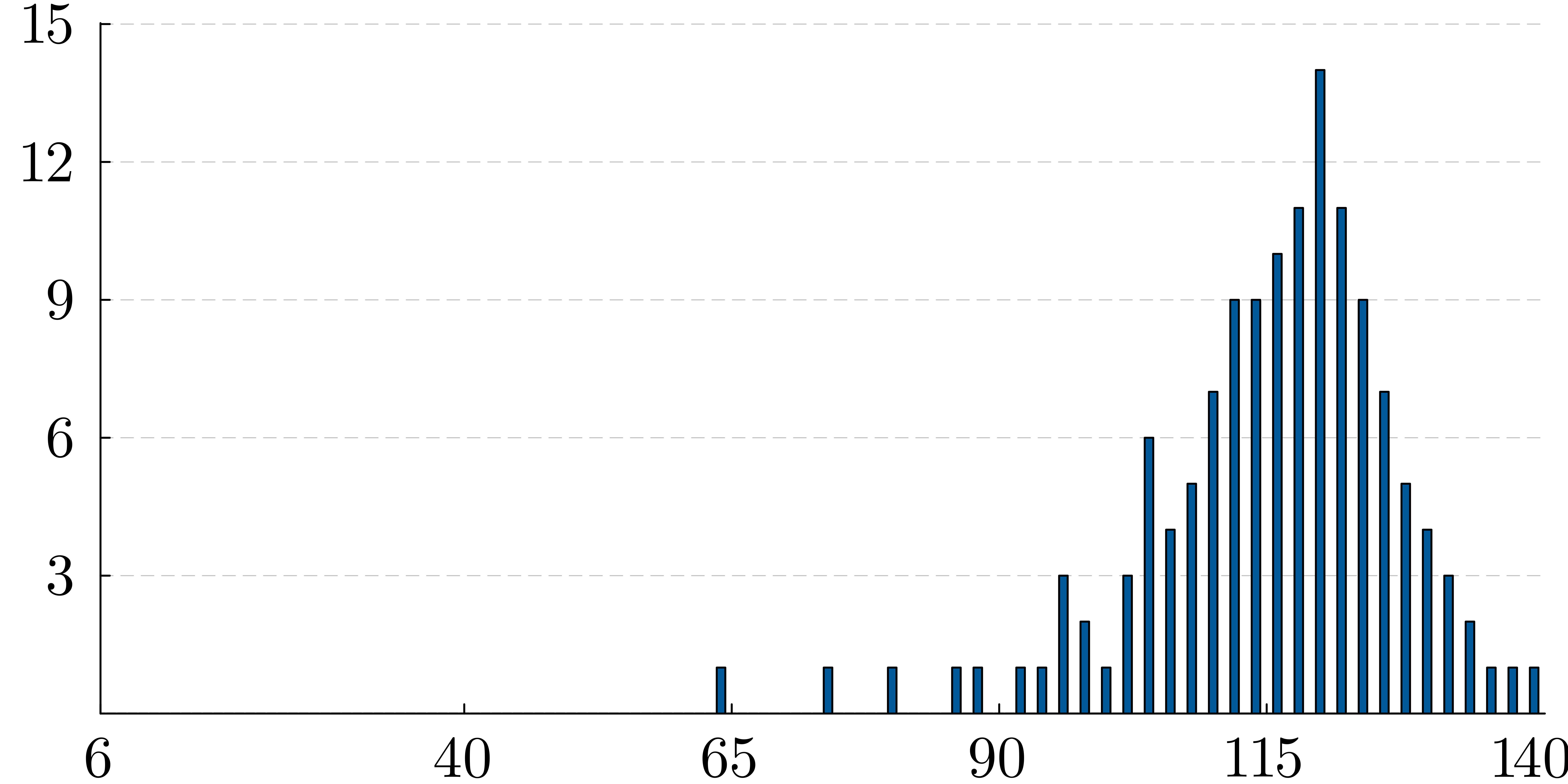}
    \caption{Number of combinatorial types of central slices of $C_7$ by number of vertices. Notice that the vertical axis has a different scaling in each plot.}
    \label{fig:distribution_7cube_all}
\end{figure}

\section{Repository}\label{app:repo}

All data produced and used in our computations are collected in a public repository \begin{center}
    \url{https://doi.org/10.5281/zenodo.17304584}
\end{center}
This appendix is intended to serve as a guide to the repository so that readers can locate and interpret all the data generated in our computations.

The repository provides both intermediate and final output data related to the computation of slices of cubes $C_d$.  
To guarantee reproducibility, we include the Julia environment files \texttt{Project.toml} and \texttt{Manifest.toml}, which specify the exact Julia version and package dependencies used in the notebook \texttt{NAGSlicesCubes.ipynb}. This notebook provides the computational workflow for \Cref{alg:all}, using the certified numerical algorithm \Cref{alg:numerical}.
 The main functions are collected in \texttt{Functions.jl}. The implementation of the exact \Cref{algo:exact} is already available at \url{https://mathrepo.mis.mpg.de/BestSlicePolytopes}.

The bulk of the repository consists of the results of the computations discussed in \Cref{sec:algo}.  
Several folders record data for each cube dimension of $d = 4,5,6,7$, with a dedicated subfolder for every dimension.  
Two sets of data correspond to different stages of the computation.

\paragraph{Preliminary computations.}
Intermediate steps of the algorithm are stored in the folder `Data-VertexTuples', which contains two types of files.
The file \texttt{dcubeVertexTuples.txt} lists all vertex tuples $\mathfrak{v}$ of $C_d$ (up to the action of the symmetry group $B_d$) that are needed to restrict the initial hyperplane arrangement, in step \ref{algo_step:vert_tuples} of \Cref{alg:numerical}. The computation of these tuples is discussed in \Cref{subsec:symmetries}.
In each file, the $i$-th line corresponds to tuples of $i+1$ linearly independent vertices (in the case of one vertex, there is a unique candidate up to $B_d$, so any choice of vertex would do the job).  
The file \texttt{dcubeCert.txt} records, for each such tuple $\mathfrak{v}$, the number of cells in the restricted hyperplane arrangement $\mathcal{H}_{A|_\mathfrak{v}}$, using the same line-by-line convention. These are the numbers stored in \Cref{tab:number_cells}.

\paragraph{Final slice data.}
The final output of \Cref{alg:all} is contained in the folder `Data-Slices', which is organized into two subfolders: `JSON', containing the files in \texttt{.mrdi} format \cite{DVJL24:mrdiFormat}, `TXT', containing the files in \texttt{.txt} format.
While the \texttt{.txt} format is easily readable, it contains numerical imprecision. On the other hand, the \texttt{.mrdi} format can be interpreted by default only with Julia (see \cite[Section 3]{DVJL24:mrdiFormat} for usage in other software systems), and produces numerical polytopes in OSCAR without numerical inconsistencies. 

Each of these two folders contains four subfolders named `4cube', `5cube', `6cube', and `7cube', corresponding to the dimension of the ambient cube.  
Within each of these subfolders, data files are organized according to the following naming scheme.  
Files ending in \texttt{@int} list all combinatorial types of generic slices (those containing no vertices of $C_d$), files ending in \texttt{@nv} list all combinatorial types of slices passing through $n$ linearly independent vertices, and files ending in \texttt{@tot} give the complete set of combinatorial types.  
The symbol \texttt{@} encodes the type of data: \texttt{f} for $f$-vectors, \texttt{u} for slice normals, and \texttt{s} for slice vertices.  
When an additional \texttt{c} follows \texttt{@} (for example \texttt{fcint}), the file contains only central slices, i.e., slices passing through the origin.

\paragraph{Graph data.}
A further component of the repository records the slices in terms of their colored vertex-edge graphs, up to color class, as discussed in \Cref{sec:graphs}.  
The folder `Graphs' contains, for each dimension $d=3,4,5$, a file \texttt{dcubeGraphVerts.txt} obtained as output of \Cref{alg:graph}.  
Every line of such a file lists the vertices of one slice of the $d$-dimensional cube, taken up to colored graph isomorphism.  
From these vertex sets one can reconstruct the colored graph of each slice by computing the convex hull to obtain the corresponding polytope, extracting its graph, and finally assigning colors to the vertices: a vertex is colored white if all its coordinates are either $+1$ or $-1$, and black otherwise.

\end{appendices}

\end{document}